\documentclass[12pt]{gtpart}
\usepackage[cp1251]{inputenc}
\usepackage[english]{babel}
\usepackage{morefloats}
\usepackage{amsmath}
\usepackage{amsfonts}
\usepackage{mathrsfs}
\usepackage{amssymb}
\usepackage{graphicx}
\usepackage{multirow}
\usepackage[perpage,para,symbol*]{footmisc}
\usepackage[usenames,dvipsnames]{color}
\newtheorem{theorem}{\textsf{Theorem}}
\newtheorem{lemma}{\textsf{Lemma}}
\newtheorem{proposition}{\textsf{Proposition}}
\def\text{\hbox} \def\newpage{\vfill\break}

\begin{document}

\title{
\date{}
{
\large \textsf{\textbf{On the optimality of the ideal right-angled $24$-cell}}
}
}
\author{\small Alexander Kolpakov}

\begin{abstract}\noindent

We prove that among four-dimensional ideal right-angled hyperbolic polytopes the $24$-cell is of minimal volume and of minimal facet number.  As a corollary, a dimension bound for ideal right-angled hyperbolic polytopes is obtained.


\end{abstract}
\maketitle

\parindent=0pt

\section{Introduction}

In this work, we consider the $24$-cell $\mathscr{C}$, that is a four-dimensional regular ideal hyperbolic polytope with Schl\"{a}fli symbol $\{3,4,3\}$ with all dihedral angles right. The polytope $\mathscr{C}$ has $24$ octahedral facets, $96$ triangular faces, $96$ edges and $24$ cubical vertex figures. We shall show that it provides a solution to the following problems in the class of ideal right-angled polytopes in $\mathbb{H}^4$:
\begin{itemize}
\item[\textbf{I:}] Find a polytope of minimal volume,
\item[\textbf{II:}] Find a polytope of minimal facet number.
\end{itemize}

Since Coxeter's work \cite{Coxeter}, the $24$-cell is known  for its nice combinatorial and geometric structure in the Euclidean sense. We demonstrate that it possesses optimal properties \textbf{I} and \textbf{II} in the hyperbolic setting. Question \textbf{I} is closely related to the volume spectrum of four-dimensional hyperbolic manifolds \cite{Ratcliffe}, question \textbf{II} is new and is both of combinatorial and geometric nature. Furthermore, using the results of \cite{Khovanskii, Nikulin}, we obtain a new dimension bound for ideal right-angled hyperbolic polytopes. The case of right-angled hyperbolic polytopes with both proper and ideal vertices was considered before in \cite{Dufour, VinbergPotyagailo}.

\medskip
\textbf{Acknowledgement.} This paper is a part of the author's Ph.D. thesis project supervised by Prof. Ruth Kellerhals. The author is grateful to her for inspiration, knowledge and support and to the referee for useful comments and suggestions. This work was supported by the Schweizerischer Nationalfonds SNF no.~200020-121506/1 and no.~200021-131967/1.

\section{Preliminaries}

Let $\mathscr{P}$ be \textit{a polytope} in the hyperbolic $n$-dimensional space $\mathbb{H}^n$, that means $\mathscr{P} = \bigcap_{i\in I} H_i^{-}$, $|I|<\infty$, with finitely many half-spaces $H_i^{-} \subset \mathbb{H}^n$ enclosed by the respective hyperplanes $H_i$. In particular, $\mathscr{P}$ is convex. The part of the polytope $\mathscr{P}$ contained in such a hyperplane is called \textit{a facet} and has dimension $n-1$. We define the low-dimensional faces of $\mathscr{P}$ by induction as non-empty intersections of the higher-dimensional ones. Let us denote the set of $k$-dimensional faces by $\Omega_k(\mathscr{P})$, $0 \leq k < n$, $\Omega_n(\mathscr{P}):=\mathscr{P}$. We call the $k$-dimensional faces of $\mathscr{P}$ for $k=0$ and $1$, its \textit{vertices} and \textit{edges}, respectively. Let $\mathrm{f}_k$ denote the number of $k$-dimensional faces and let $\mathbf{f}(\mathscr{P}) = (\mathrm{f}_0, \cdots, \mathrm{f}_{n-1})$ be \textit{the face vector} of the polytope $\mathscr{P}$.

The faces of a given $n$-polytope $\mathscr{P}$ form a lattice $\Omega(\mathscr{P}) := \bigcup_{k=0}^{n} \Omega_k(\mathscr{P})$ ordered by the set-theoretical face inclusion. Call two polytopes $\mathscr{P}_1$ and $\mathscr{P}_2$ combinatorially isomorphic if their face lattices $\Omega(\mathscr{P}_1)$ and $\Omega(\mathscr{P}_2)$ are isomorphic. Call $\mathscr{P} \subset \mathbb{H}^n$ \textit{a regular} hyperbolic polytope if it is combinatorially isomorphic to a regular $n$-dimensional Euclidean polytope and all the dihedral angles of $\mathscr{P}$ in its co-dimension two faces are equal. Recall that there are infinitely many regular polygons. Dimension three provides five Platonic solids. There exist six regular four-dimensional polytopes. Starting from dimension five, there are only three combinatorial types of convex regular polytopes (see \cite[Table~I]{Coxeter}).

A polytope is called \textit{non-obtuse} if all the dihedral angles in its co-dimension two faces are not greater than $\pi/2$. A polytope is \textit{right-angled} if all these dihedral angles equal $\pi/2$. Call a hyperbolic polytope $\mathscr{P}$ \textit{ideal} if all its vertices are ideal points of $\mathbb{H}^n$, i.e. all of them belong to $\partial \overline{\mathbb{H}^n}$. A polytope $\mathscr{P} \subset \mathbb{H}^n$ is \textit{simple} if each vertex belongs to $n$ facets only, and $\mathscr{P}$ is called \textit{simple at edges} if each edge belongs to $n-1$ facets only. An infinitesimal link (w.r.t. the Euclidean metric on $\overline{\mathbb{H}^n}$) of a vertex of a polytope is called its \textit{vertex figure} (w.r.t. the given vertex). Every vertex figure of a compact non-obtuse hyperbolic polytope is a combinatorial simplex of co-dimension one \cite[p.~108, Theorem~1.8]{Vinberg}. Every vertex figure of an ideal right-angled hyperbolic polytope is a combinatorial cube of co-dimension one \cite[Proposition~1]{Dufour}. Thus, a compact non-obtuse hyperbolic polytope is simple and an ideal right-angled hyperbolic polytope is simple at edges.

The $24$-cell considered as a regular polytope has the Schl\"{a}fli symbol $\{3,4,3\}$, octahedral facets $\{3,4\}$ and cubical vertex figures $\{4,3\}$. Note that this is the only regular four-dimensional polytope having each vertex figure a cube \cite[Table~I]{Coxeter} and thus realisable as an ideal right-angled hyperbolic one. We denote it by $\mathscr{C}$ and call \textit{the hyperbolic $24$-cell}.

\section{The $24$-cell and volume minimality}

\begin{lemma}[Combinatorial identities]\label{lemma_combinatorics}
Let $\mathscr{P} \subset \mathbb{H}^4$ be an ideal right-angled polytope with face vector $\mathbf{f}(\mathscr{P}) = (\mathrm{f}_0, \mathrm{f}_1, \mathrm{f}_2, \mathrm{f}_3)$. Then the following combinatorial identities hold.

\begin{equation}\label{eqEuler}
\mathrm{f}_0 - \mathrm{f}_1 + \mathrm{f}_2 - \mathrm{f}_3 = 0,
\end{equation}
\begin{equation}\label{eqVertexEdge}
\mathrm{f}_1 = 4\, \mathrm{f}_0,
\end{equation}
\begin{equation}\label{eqVertexSum}
12\, \mathrm{f}_0 = \sum_{F \in \Omega_2(\mathscr{P})} \mathrm{f}_0(F).
\end{equation}
\end{lemma}
\begin{proof}
We list the proofs of (\ref{eqEuler})-(\ref{eqVertexSum}) below in the respective order.

\medskip
\textrm{(1)} This is Euler's identity. Since $\mathscr{P}$ is a convex four-dimensional polytope, its surface $\partial\mathscr{P}$ is homeomorphic to $\mathbb{S}^3$. Hence, for the Euler characteristic of $\partial\mathscr{P}$, we have $\mathrm{f}_0 - \mathrm{f}_1 + \mathrm{f}_2 - \mathrm{f}_3 =: \chi(\partial\mathscr{P}) = \chi(\mathbb{S}^3) = 0$.

\medskip
\textrm{(2)} Let $v \in \Omega_0(\mathscr{P})$ be a vertex. Then each vertex figure $\mathscr{P}_v$ is a cube, since $\mathscr{P}$ is an ideal right-angled polytope \cite[Proposition~1]{Dufour}. The vertices of $\mathscr{P}_v$ correspond to the edges of $\mathscr{P}$ emanating from a given vertex $v\in \Omega_0(\mathscr{P})$. This means that eight edges are adjacent at $v$. On the other hand, each edge has two vertices. Thus, we obtain $2\,\mathrm{f}_1 = 8\,\mathrm{f}_0$ and (\ref{eqVertexEdge}) follows.

\medskip
\textrm{(3)} The edges of the vertex figure $\mathscr{P}_v$ correspond to the two-dimensional faces of $\mathscr{P}$ meeting $v$. Thus, twelve two-dimensional faces meet at each vertex. Hence, if we sum up the number of vertices $\mathrm{f}_0(F)$ over all the two-dimensional faces $F\in\Omega_2(\mathscr{P})$, we count each vertex of $\mathscr{P}$ twelve times. Then the desired formula follows and the lemma is proven.
\end{proof}

\begin{lemma}[Volume formula]\label{lemma_volume}
Let $\mathscr{P} \subset \mathbb{H}^4$ be an ideal right-angled polytope with face vector $\mathbf{f}(\mathscr{P}) = (\mathrm{f}_0, \mathrm{f}_1, \mathrm{f}_2, \mathrm{f}_3)$. Then its volume equals
\begin{equation*}
\mathrm{Vol}\,\mathscr{P} = \frac{\mathrm{f}_0 - \mathrm{f}_3 + 4}{3}\, \pi^2.
\end{equation*}
\end{lemma}
\begin{proof}
Let $\mathscr{G}$ be the group generated by reflections in the supporting hyperplanes of the facets of $\mathscr{P}$. The group $\mathscr{G}$ is a Coxeter group acting discretely on $\mathbb{H}^4$ with Poincar\'{e} polytope $\mathscr{P}$. Hence $\mathrm{Vol}\,\mathscr{P} = \mathrm{coVol}\,\mathscr{G}$. 
By \cite[Corollary~4.2]{Zehrt}, we obtain 
\begin{equation}\label{eqZehrt}
\mathrm{coVol}\,\mathscr{G} = \frac{\pi^2}{3} \left( \kappa(\mathscr{P}) - 2 \sum_{F\in \Omega_2(\mathscr{P})} \frac{\mathrm{f}_0(F) - 2}{|Stab(F)|} + 4 \sum_{v\in \Omega_0(\mathscr{P})} \frac{1}{|Stab(v)|} \right),
\end{equation}
where $\kappa(\mathscr{P}) := 4 - 2 (\mathrm{f}_0 + \mathrm{f}_2) + \sum_{F\in \Omega_2(\mathscr{P})} \mathrm{f}_0(F)$.
Here $Stab(\star) \leqslant \mathscr{G}$ means the stabilizer of $\star$, where $\star$ is a two-dimensional face $F$ or a vertex $v$ of $\mathscr{P}$.

\smallskip
For each $F\in\Omega_2(\mathscr{P})$ the group $Stab(F)$ is generated by two reflections in the facets meeting at $F$. Since $\mathscr{P}$ is a right-angled polytope, the group $Stab(F)$ is generated by reflections in two orthogonal mirrors. Thus,  $Stab(F)\simeq D_2$, the dihedral group of order four. Each vertex $v\in\Omega_0(\mathscr{P})$ is ideal, i.e. belongs to $\partial \overline{\mathbb{H}^4}$, so the stabilizer $Stab(v)$ is a Coxeter group generated by reflection in the faces of a Euclidean cube. Hence $Stab(v)$ is infinite.

\smallskip
We obtain the desired volume formula by substituting the values $|Stab(F)| = 4$, $|Stab(v)| = \infty$ into (\ref{eqZehrt}) and by applying the formulas of Lemma~\ref{lemma_combinatorics} for the computation.
\end{proof}

\medskip
The hyperbolic $24$-cell $\mathscr{C}$ has $\mathrm{f}_0=\mathrm{f}_3=24$, $\mathrm{f}_1=\mathrm{f}_2=96$, see \cite[Table~I, (ii)]{Coxeter}. Hence, by the lemma above, its volume equals $4\pi^2/3$.

\begin{theorem}[Minimal volume]\label{theorem_min_volume}
A four-dimensional ideal right-angled hyperbolic polytope of minimal volume is $\mathscr{C}$, up to an isometry.
\end{theorem}
\begin{proof}
Let us consider an ideal right-angled hyperbolic polytope $\mathscr{P} \subset \mathbb{H}^4$. Let $\mathrm{f}_2(k)$ denote the number of its two-dimensional $k$-gonal faces, $k \geq 3$, which are ideal hyperbolic polygons. Then
\begin{equation*}
\mathrm{f}_2 = \mathrm{f}_2(3) + \dots + \mathrm{f}_2(N),
\end{equation*}
where $N = \max_{F \in \Omega_2(\mathscr{P})} \mathrm{f}_0(F) \geq 3$.
By Lemma~\ref{lemma_combinatorics}, formula (\ref{eqVertexSum}), we obtain
\begin{equation*}
12\, \mathrm{f}_0 = \sum_{F \in \Omega_2(\mathscr{P})}\mathrm{f}_0(F) = 3\, \mathrm{f}_2(3) + \dots + N\, \mathrm{f}_2(N).
\end{equation*}
By using Lemma~\ref{lemma_combinatorics}, formulas (\ref{eqEuler})-(\ref{eqVertexEdge}), one subsequently computes
\begin{equation}\label{eqf0f3}
\mathrm{f}_0 - \mathrm{f}_3 = 4 \mathrm{f}_0 - \mathrm{f}_2 = \frac{1}{3} \sum^{N}_{k=4}\,(k-3)\mathrm{f}_2(k) \geq 0.
\end{equation}
Then, by Lemma~\ref{lemma_volume},
\begin{equation*}
\mathrm{Vol}\,\mathscr{P} \geq \frac{4}{3}\,\pi^2 = \mathrm{Vol}\,\mathscr{C}.
\end{equation*}

If $\mathrm{Vol}\,\mathscr{P}$ equals the volume of $\mathscr{C}$, one immediately has
$\mathrm{f}_2(k) = 0$ for all $k \geq 4$ by (\ref{eqf0f3}). This means that all the two-dimensional faces of $\mathscr{P}$ are triangles. Consider a facet $P \in \Omega_3(\mathscr{P})$. Observe that $P \subset \mathbb{H}^3$ is an ideal right-angled polyhedron which has only triangular faces. Then $P$ is a combinatorial octahedron and it is isometric to the right-angled hyperbolic octahedron by \cite[Theorem~3]{Andreev1} and \cite[Theorem~2]{Andreev2}. Hence, all the facets of $\mathscr{P}$ are ideal right-angled octahedra. So the polytope $\mathscr{P}$ is combinatorially isomorphic to a regular four-dimensional Euclidean polytope with octahedral facets only, that is, the $24$-cell by \cite[Table~I, (ii)]{Coxeter}. Thus $\mathscr{P}$ is isometric to $\mathscr{C}$ by Andreev's theorem \cite[Theorem~3]{Andreev1}.
\end{proof}

\section{The $24$-cell and facet number minimality}

\begin{theorem}[Minimal facet number]\label{theorem_min_facet}
The facet number of a four-dimensional ideal right-angled hyperbolic polytope $\mathscr{P}$ satisfies the inequality $\mathrm{f}_3(\mathscr{P}) \geq \mathrm{f}_3(\mathscr{C}) = 24$. Any four-dimensional ideal right-angled hyperbolic polytope $\mathscr{P}$ with $\mathrm{f}_3(\mathscr{P}) = 24$ is isometric to the hyperbolic $24$-cell $\mathscr{C}$.
\end{theorem}

The proof will be based on Proposition \ref{proposition_flower} and Lemma \ref{lemma_circuit} below. Their proofs will be given in Section~3.3.

\medskip
\textbf{3.1 Three-dimensional ideal right-angled hyperbolic polyhedra with few faces.} Let $\mathscr{A}_k \subset \mathbb{H}^3$, $k \geq 3$, be an ideal right-angled antiprism depicted in Fig.~\ref{antiprism_k}. In the figure, the leftmost and the rightmost edges are identified, so that the surface of the polyhedron is partitioned into top and bottom $k$-gonal faces and $2k$ triangular faces in the annulus between them. Such an antiprism exists for every $k \geq 3$ and it is unique up to an isometry due to \cite[Theorem~3]{Andreev1}, \cite[Theorem~2]{Andreev2}. 

\smallskip
Antiprisms $\mathscr{A}_k$ will later play the r\^{o}le of possible facets for a four-dimensional ideal right-angled hyperbolic polytope in the proof of Theorem~\ref{theorem_min_facet}.

\begin{proposition}[Antiprism's optimality]\label{proposition_flower}
A three-dimensional ideal right-angled hyperbolic polyhedron of minimal facet number, which has at least one $k$-gonal face, $k \geq 3$, is isometric to the antiprism $\mathscr{A}_k$ with $\mathrm{f}_2(\mathscr{A}_k) = 2k+2$.
\end{proposition}
\begin{proof}
Let $\mathscr{P} \subset \mathbb{H}^3$ be an ideal right-angled polyhedron. Let $F \in \Omega_2(\mathscr{P})$ be a $k$-gonal face, $k \geq 3$. For each edge $e\in \Omega_1(F)$ there is exactly one further face adjacent to $F$ along $e$. For each vertex $v$, being four-valent by Andreev's theorem \cite[Theorem~2]{Andreev2}, there exists a face intersecting $F$ at $v$ only. Moreover, all the faces mentioned above are different from each other, so that we have $\mathrm{f}_2(\mathscr{P}) \geq 2k+1$. Observe that these faces can not constitute yet a polyhedron. Indeed, consider $F$ as a ``bottom'' face of $\mathscr{P}$. Then the new faces we have added make a surface wrapping around the interior of $\mathscr{P}$ along the edges of $F$. Since all vertices are four-valent, at least one additional ``top'' face is required to close up the polyhedron.  Hence $\mathrm{f}_2(\mathscr{P}) \geq 2k+2$. The antiprism $\mathscr{A}_k$ satisfies 
\begin{equation}\label{eqAntiprismFace}
\mathrm{f}_2(\mathscr{A}_k) = 2k+2
\end{equation}
and so has minimal facet number.

\smallskip
It remains to show that a polyhedron $\mathscr{P}$ with $\mathrm{f}_2(\mathscr{P}) = \mathrm{f}_2(\mathscr{A}_k)$ is in fact isometric to $\mathscr{A}_k$. Since $\mathscr{P}$ has four-valent vertices, $2 \mathrm{f}_1(\mathscr{P}) = 4 \mathrm{f}_0(\mathscr{P})$. From this equality and Euler's identity $\mathrm{f}_0(\mathscr{P})-\mathrm{f}_1(\mathscr{P})+\mathrm{f}_2(\mathscr{P}) = 2$ we obtain that
\begin{equation}\label{eqAntiprismf0f2}
\mathrm{f}_2(\mathscr{P}) = \mathrm{f}_0(\mathscr{P}) + 2.
\end{equation}

\smallskip
Consider the faces adjacent to the $k$-gon $F$ along its edges. We shall prove that no pair of them can have a common vertex $v \notin \Omega_0(\mathscr{P})$. By supposing the contrary, let us denote two such faces $F_i$, $i=1,2$, and let them intersect at $v$. Observe that $F_i$, $i=1,2$, are adjacent to $F$ along two disjoint edges $e_1$ and $e_2$. In fact, if $e_1$ intersects $e_2$ in a vertex $u\in\Omega_0(F)$, then since $\mathscr{P}$ has convex faces we obtain two geodesic segments joining $v$ to $u$. One of them belongs to $F_1$ and the other belongs to $F_2$. This is impossible, unless the considered segments are represented by a common edge $e$ of $F_i$, $i=1,2$, adjacent to both $v$ and $u$. But then the vertex $u$ has only three adjacent edges: $e_1$, $e_2$ and $e$. This is a contradiction to $u$ having valency four. Now if $F_1$ and $F_2$ share an edge $e$ such that $v \in \Omega_0(\mathscr{P})$, then condition ($\mathfrak{m}_2$) of Andreev's theorem \cite[Theorem 2]{Andreev2} does not hold as depicted in Fig.~\ref{antiprism3circuit}. If $F_1$ and $F_2$ share only the vertex~$v$, then condition ($\mathfrak{m}_5$) of \cite[Theorem 2]{Andreev2} is not satisfied as depicted in Fig.~\ref{antiprism2circuit}.

\smallskip
Suppose that a face $F^{\prime}$ adjacent to the $k$-gon $F \in \Omega_2(\mathscr{P})$ along an edge is not triangular. Then $F^{\prime}$ has $\mathrm{f}_0(F^{\prime})$ vertices, and two among them are already counted in $\mathrm{f}_0(F)$. Hence we have at least 
\begin{equation*}
\sum_{\parbox[h]{2.2cm}{\scriptsize{$F^{\prime}$ adjacent to\\ $F$ along an edge}}} \left(\mathrm{f}_0(F^{\prime}) - 2\right) \geq (k-1)+2 = k+1
\end{equation*}
additional vertices, since $\mathrm{f}_0(F^{\prime}) \geq 3$ for each $F^{\prime}$ among $k$ faces adjacent to $F$ and at least one such face $F^{\prime}$ has $\mathrm{f}_0(F^{\prime}) \geq 4$. Thus $\mathrm{f}_0(\mathscr{P}) \geq 2k+1$, and by (\ref{eqAntiprismf0f2}) the estimate $\mathrm{f}_2(\mathscr{P}) \geq 2k+3$ follows. Equality (\ref{eqAntiprismFace}) implies $\mathrm{f}_2(\mathscr{P}) > \mathrm{f}_2(\mathscr{A}_k)$ and we arrive at a contradiction. 
Hence all the faces adjacent to $F$ along its edges are triangular.

\smallskip
Consider the faces of $\mathscr{P}$ adjacent to the $k$-gon $F \in \Omega_2(\mathscr{P})$ only at its vertices. Suppose that one of them, say $F^{\prime}$, is not triangular. Then we have 
\begin{equation*}
\sum_{\parbox[h]{2.2cm}{\scriptsize{$F^{\prime}$ adjacent to\\ $F$ at a vertex}}} \mathrm{f}_1(F^{\prime}) \geq 3(k-1)+4 = 3k+1
\end{equation*}
additional edges. But then $\mathrm{f}_1(\mathscr{P}) \geq 4k+1$ and we arrive at a contradiction. Indeed, in this case $\mathrm{f}_1(\mathscr{P}) > \mathrm{f}_1(\mathscr{A}_k) = 4k$, while a polyhedron of minimal facet number has all $\mathrm{f}_i$ minimal, $i=0,1,2$, as follows from equation (\ref{eqAntiprismf0f2}) and 
the remarks before.

\smallskip
Hence we have a $k$-gonal face $F$, $k\geq 3$, together with $2k$ triangular side faces adjacent to it along the edges and at the vertices. By adding another one $k$-gonal face we close up the polyhedron $\mathscr{P}$, while its vertex number remains unchanged. Observe that there is no other way to finish this construction without increasing the vertex number.

\smallskip
Thus, an ideal right-angled polyhedron $\mathscr{P} \subset \mathbb{H}^3$ having minimal face number, which contains at least one $k$-gon, is combinatorially isomorphic to $\mathscr{A}_k$. By \cite[Theorem~3]{Andreev1} and \cite[Theorem~2]{Andreev2}, the polyhedron $\mathscr{P}$ is isometric to $\mathscr{A}_k$.
\end{proof}

\medskip
\textsf{\textbf{Note (to Proposition~\ref{proposition_flower})}} The classification of \textit{polygonal maps} on the two-dimensional sphere given in \cite{Deza2} provides another argument to show the uniqueness of antiprism stated above. Namely, \cite[Theorem~1]{Deza2} says that $\mathscr{P}$ has in fact not less than two $k$-gonal faces. Hence $\mathrm{f}_2(\mathscr{P}) = 2k+2$ if and only if $\mathscr{P}$ has exactly two $k$-gonal faces and $2k$ triangular faces. Polygonal maps of this kind are classified by \cite[Theorem~2]{Deza2}. Among them only the map isomorphic to the one-skeleton of $\mathscr{A}_k$ satisfies Steiniz's theorem \cite[Chapter~4]{Ziegler}. Thus, the polyhedron $\mathscr{P}$ is combinatorially isomorphic to $\mathscr{A}_k$.\footnote{the author is grateful to Michel Deza for indicating the very recent paper \cite{Deza2}.}

\medskip
\textbf{3.2 Combinatorial constraints on facet adjacency.} Let $F_1$, $\dots$, $F_m$ be an ordered sequence of facets of a given hyperbolic polytope $\mathscr{P} \subset \mathbb{H}^4$ such that each facet is adjacent only to the previous and the following ones either through a co-dimension two face or through an ideal vertex, while the last facet $F_m$ is adjacent only to the first facet $F_1$ (through a co-dimension two face or through an ideal vertex, as before) and no three of them share a lower-dimensional face. Call the sequence $F_1$, $\dots$, $F_m$ \textit{a $(k,\ell)$ circuit}, $k+\ell = m$, if it comprises $k$ co-dimension two faces and $\ell$ ideal vertices shared by the facets. We complete the analysis carried out in \cite{VinbergPotyagailo} in the following way.

\begin{lemma}[Adjacency constraints]\label{lemma_circuit}
Let $\mathscr{P} \subset \mathbb{H}^4$ be an ideal right-angled polytope. Then $\mathscr{P}$ contains no $(3,0)$, $(4,0)$ and $(2,1)$ circuits.
\end{lemma}
\begin{proof}
By \cite[Proposition~4.1]{VinbergPotyagailo} there are no $(3,0)$ and $(2,1)$ circuits. Suppose on the contrary that there exists a $(4,0)$ circuit formed by the facets $F_k\in \Omega_3(\mathscr{P})$, $k=1,2,3,4$. Let $e_k$, $k=1,2,3,4$, denote the outer unit vector normal to the support hyperplane of $F_k$. Consider the Gram matrix of these vectors w.r.t. the Lorentzian form $\langle\cdot, \cdot\rangle_{4,1}$:
\begin{equation*}
G = \left( \langle e_i, e_j \rangle \right)^{4}_{i, j=1} = 
\left(
\begin{array}{cccc}
1& 0& -\cosh\rho_{13}& 0\\
0& 1& 0& -\cosh\rho_{24}\\
-\cosh\rho_{13}& 0& 1& 0\\
0& -\cosh\rho_{24}& 0& 1
\end{array}
\right),
\end{equation*}
where $\rho_{ij} > 0$ is the length of the common perpendicular between two disjoint support hyperplanes for $F_i$ and $F_j$ respectively.
The eigenvalues of $G$ are $\{1\pm\cosh\rho_{13}, 1\pm\cosh\rho_{24}\}$, that means two of them are strictly negative and two are strictly positive. Thus, we arrive at a contradiction with the signature of a Lorentzian form.
\end{proof}

\medskip
\textbf{3.3 Proof of Theorem~\ref{theorem_min_facet}.} Let $\mathscr{P} \subset \mathbb{H}^4$ be an ideal right-angled polytope. Let $P \in \Omega_3(\mathscr{P})$ be a facet. For every two-face $F \in \Omega_2(P)$ there exists a corresponding facet $P^{\prime} \in \Omega_3(\mathscr{P})$, $P^{\prime} \neq P$, such that $P$ and $P^{\prime}$ share the face $F$. Since each vertex figure of $\mathscr{P}$ is a cube, there exists a respective facet $P^{\prime\prime} \in \Omega_3(\mathscr{P})$ for every vertex $v \in \Omega_0(P)$. The vertex figure is depicted in Fig.~\ref{vertex}, where the grey bottom face of the cube corresponds to $P$ and the top face corresponds to $P^{\prime\prime}$. These new facets $P^{\prime}$ and $P^{\prime\prime}$ together with $P$ are pairwise different. In order to show this we use the following convexity argument.

\medskip
(\textit{Convexity argument})
First, observe that no facet of a convex polytope can meet another one at two different two-faces. Now suppose that $P^{\prime}\in\Omega_3(\mathscr{P})$ is a facet adjacent to $P$ at a face $F \in \Omega_2(P)$ and a single vertex $v \in \Omega_0(P)$ not in $F$. The facets $P$ and $P^{\prime}$ have non-intersecting interiors, but the geodesic going through a given point of $F$ to $v$ belongs to both of them by the convexity of $\mathscr{P}$. So we arrive at a contradiction. 

\smallskip
The same contradiction arises if we suppose that there is a facet $P^{\prime}\in\Omega_3(\mathscr{P})$ adjacent to $P$ at two distinct vertices $v, v^{\prime} \in \Omega_0(P)$. In this case we consider the geodesic in $P$ going through $v$ to $v^{\prime}$.

\medskip
By the convexity argument above, the facet number of $\mathscr{P}$ has the lower bound
\begin{equation*}
\mathrm{f}_3(\mathscr{P}) \geq \mathrm{f}_2(P) + \mathrm{f}_0(P) + 1,
\end{equation*}
or, by means of equality (\ref{eqAntiprismf0f2}),
\begin{equation}\label{eq25}
\mathrm{f}_3(\mathscr{P}) \geq 2\, \mathrm{f}_2(P) - 1.
\end{equation}

\smallskip
Observe that the hyperbolic $24$-cell $\mathscr{C}$ has only triangle two-faces. Suppose that $\mathscr{P}$ has at least one $k$-gonal face $F \in \Omega_2(\mathscr{P})$ with $k\geq 4$. We shall show that the estimate $\mathrm{f}_3(\mathscr{P}) \geq 25$ holds, by considering several cases as follows. 

\medskip
\textbf{A)} Suppose that $\mathscr{P}$ has a $k$-gonal two-dimensional face with $k \geq 6$. Then, by~\eqref{eq25} and Proposition~\ref{proposition_flower}, we have 
\begin{equation*}
\mathrm{f}_3(\mathscr{P}) \geq 2\, \mathrm{f}_2(\mathscr{A}_k) - 1 = 2 (2k+2) - 1 \geq 27.
\end{equation*}
Thus $\mathscr{P}$ is not of minimal facet number.

\medskip
\textbf{B)} Suppose that $\mathscr{P}$ has a pentagonal two-dimensional face $F$ contained in a facet $P \in \Omega_3(\mathscr{P})$. Suppose $P$ is not isometric to $\mathscr{A}_5$. This assumption implies $\mathrm{f}_2(P) > 12$. Then \eqref{eq25} grants $\mathrm{f}_3(\mathscr{P}) \geq 25$. 

\medskip
\textbf{C)} Suppose that all the facets of $\mathscr{P}$ containing a pentagonal two-face are isometric to $\mathscr{A}_5$. Let $P_0$ be one of them. Then it  has two neighbouring facets $P_k$, $k=1,2$ both isometric to $\mathscr{A}_5$. Now we count the facets adjacent to $P_k$, $k=0,1,2$ in Fig.~\ref{antiprism5proof}, where $P_0$ is coloured grey. Observe that two-faces in Fig.~\ref{antiprism5proof} sharing an edge are marked with the same number and belong to a common facet, since $\mathscr{P}$ is simple at edges. However, the two-faces marked with different numbers, correspond to different adjacent facets. Suppose on the contrary that there are two faces $F \in \Omega_2(P_i)$, $F^{\prime} \in \Omega_2(P_j)$, $i,j \in \{0,1,2\}$, marked with distinct numbers and a facet $P^{\prime} \in \Omega_3(\mathscr{P})$ such that $P^{\prime}$ is adjacent to $P_i$ at $F$ and to $P_j$ at $F^{\prime}$ and consider the following cases.

\medskip
\textbf{C.1)} If $i=j$, we arrive at a contradiction by the convexity argument above.

\medskip
\textbf{C.2)} If $i=0$, $j\in\{1,2\}$, then there exists a unique geodesic joining a point $p$ of $F$ to a point $p^{\prime}$ of $F^{\prime}$. Observe in Fig.~\ref{antiprism5proof}, that the point $p^{\prime}$ may be chosen so that $p^{\prime} \in F^{\prime} \cap P_0$. Then the geodesic between $p$ and $p^{\prime}$ intersects both the interior of $P^{\prime}$ and the interior of $P_0$. Again, we use the convexity argument and arrive at a contradiction.

\medskip
\textbf{C.3)} Let $i=1$, $j=2$. Then if there exist a face $\tilde{F} \in \Omega_2(P_0)$, $\tilde{F} \cap F \neq \emptyset$, and a face $\tilde{F^{\prime}} \in \Omega_2(P_0)$, $\tilde{F^{\prime}} \cap F^{\prime} \neq \emptyset$, we reduce our argument to case \textbf{C.1} by considering a geodesic segment joining a point of $\tilde{F} \cap F$ to a point of $\tilde{F^{\prime}} \cap F^{\prime}$.

\medskip
The only case when no such two faces $\tilde{F}$ and $\tilde{F^{\prime}}$ exist is if $F$ has number $21$ and $F^{\prime}$ has number $22$ in Fig.~\ref{antiprism5proof}. Then the $(4,0)$-circuit $P_0P_1P^{\prime}P_2$ appears, in contrary to Lemma~\ref{lemma_circuit}.

\smallskip
Thus, one has $22$ new facets adjacent to $P_k$, $k=0,1,2$. Together with $P_k$ themselves, $k=0,1,2$, they provide $\mathrm{f}_3(\mathscr{P}) \geq 25$.

\medskip
\textbf{D)} By now, cases \textbf{A}, \textbf{B} and \textbf{C} imply that if an ideal right-angled hyperbolic polytope $\mathscr{P} \subset \mathbb{H}^4$ has at least one $k$-gonal face with $k\geq 5$, then $\mathrm{f}_3(\mathscr{P}) \geq 25$. Suppose that $\Omega_2(\mathscr{P})$ contains only triangles and quadrilaterals. 

\smallskip
By Andreev's theorem \cite{Andreev2}, each facet $P \in \Omega_3(\mathscr{P})$ has only four-valent vertices. By assumption, $P$ has only triangular and quadrilateral faces. Combinatorial polyhedra of this type are introduced in \cite{Deza1} as \textit{octahedrites} and the list of those possessing up to 17 vertices is given. Note that in view of (\ref{eq25}) we may consider octahedrites that have not more than twelve faces or, by equality (\ref{eqAntiprismf0f2}) from Proposition~\ref{proposition_flower}, ten vertices. In Fig.~\ref{octahedrite89}, \ref{octahedrite10} we depict only those realisable as ideal right-angled hyperbolic polyhedra with eight, nine and ten vertices. The ideal right-angled octahedron has six vertices and completes the list. By considering each of the polyhedra in Fig.~\ref{octahedrite89} and Fig.~\ref{octahedrite10} as a possible facet $P \in \Omega_3(\mathscr{P})$, we shall derive the estimate $\mathrm{f}_3(\mathscr{P}) \geq 25$.

\medskip
\textbf{D.1)} Let $P_0 \in \Omega_3(\mathscr{P})$ be the hyperbolic octahedrite with ten vertices depicted in Fig.~\ref{octahedrite10proof}. Consider the facets of $\mathscr{P}$ adjacent to $P_0$ at its faces. One has $\mathrm{f}_2(P_0) = 12$, and hence $\mathrm{f}_3(\mathscr{P}) \geq 12$. Consider the faces coloured grey in Fig.~\ref{octahedrite10proof}: the front face is called $F_1$ and the back face, called $F_2$, is indicated by the grey arrow.

\smallskip
The facets $P_1, P_2 \in \Omega_3(\mathscr{P})$ adjacent to $P_0$ at $F_1$ and $F_2$, respectively, contain quadrilaterals among their faces. By Proposition~\ref{proposition_flower}, it follows that $\mathrm{f}_2(P_i) \geq \mathrm{f}_2(\mathscr{A}_4) = 10$, $i=1,2$. We shall count all new facets $P^{\prime}$ brought by face adjacency to $P_i$, $i=1,2$.

\smallskip
Observe that no $P^{\prime}$, which does not share an edge with $P_0$, can be adjacent simultaneously to $P_i$ and $P_j$, $i,j \in \{1,2\}$, at two-faces, since otherwise the $(4,0)$ circuit $P_1 P_0 P_2 P^{\prime}$ appears in contrary to Lemma~\ref{lemma_circuit}. 

\smallskip
Each facet $P^{\prime}$ that shares an edge with $F_k$, $k=1,2$, is already counted as adjacent to $P_0$. The facets $P_1$ and $P_2$ are already counted as well, by the same reason. Then the total number of new facets coming together with $P_1$ and $P_2$ is at least $\sum^2_{i=1}\mathrm{f}_2(P_i) - \sum^2_{i=1}\mathrm{f}_1(F_i) - 2  \geq 2\cdot 10 - 2\cdot 4 - 2 = 10$. This implies the estimate $\mathrm{f}_3(\mathscr{P}) \geq 12 + 10 = 22$.

\smallskip
Consider the facets $P_i$, $i=3,4$, adjacent to $P_0$ only at the corresponding circumscribed grey vertices $v_i$, $i=3,4$, in Fig.~\ref{octahedrite10proof}. Then consider the case if $P^{\prime}$ is adjacent to $P_j$, $j \in \{1,2\}$ at a two-face $F^{\prime} \in \Omega_2(P_j)$. If there exist a face $\tilde{F^{\prime}} \in \Omega_2(P_0)$ such that $F^{\prime} \cap \tilde{F^{\prime}} \neq \emptyset$, then choose a point $p \in F^{\prime} \cap \tilde{F^{\prime}}$ and use the convexity argument again for the geodesic going through $p$ to $v_i$. If $F^{\prime} \cap \tilde{F^{\prime}} = \emptyset$, then the $(2,1)$ circuit $P_0P_1P^{\prime}$ appears in contrary to Lemma~\ref{lemma_circuit}. Adding up two new facets gives $\mathrm{f}_3(\mathscr{P}) \geq 24$. Finally, we count $P_0$ itself and arrive at the estimate $\mathrm{f}_3(\mathscr{P}) \geq 25$.

\medskip
\textbf{D.2)} Let $P_0 \in \Omega_3(\mathscr{P})$ be the hyperbolic octahedrite with nine vertices and eleven faces depicted on the right in Fig.~\ref{octahedrite89}. Consider the facets adjacent to $P_0$ at its two-dimensional faces. By counting them, we have $\mathrm{f}_3(\mathscr{P}) \geq \mathrm{f}_2(P_0) = 11$.

\smallskip
Consider the facet $P_1$ adjacent to the triangle face $F_1$ of $P_0$ coloured grey in the center of Fig.~\ref{octahedrite9proof}. By Proposition~\ref{proposition_flower}, we have $\mathrm{f}_2(P_1) \geq \mathrm{f}_2(\mathscr{A}_3) = 8$. By excluding already counted facets adjacent to $P_0$ like in case \textbf{D.1}, the facet $P_1$ brings new $\mathrm{f}_2(P_1) - \mathrm{f}_1(F_1) - 1 \geq 8 -3 -1 = 4$ ones by face adjacency. Then $\mathrm{f}_3(\mathscr{P}) \geq 15$. The visible part of the facet $P_2$ adjacent to $P_0$ at its back face $F_2$ is coloured grey in Fig.~\ref{octahedrite9proof}. Again, we have $\mathrm{f}_2(P_2) \geq \mathrm{f}_2(\mathscr{A}_3) = 8$. By counting new facets adjacent to $P_2$ at faces, it brings another $\mathrm{f}_2(P_2) - \mathrm{f}_1(F_2) - 1 \geq 8 -3 -1 = 4$ new ones. Hence $\mathrm{f}_3(\mathscr{P}) \geq 19$.

\smallskip
The facets $\widehat{P}_k$, $k=3,4,5$, adjacent to $P_0$ only at the circumscribed hollow vertices $v_k$, $k=3,4,5$, in Fig.~\ref{octahedrite9proof} are different from the already counted ones either by the convexity argument or by Lemma~\ref{lemma_circuit}, which forbids $(2,1)$ circuits, c.f. the argument of case \textbf{D.1}. Thus $\mathrm{f}_3(\mathscr{P}) \geq 22$. 

\smallskip
Let $\widehat{P}_k$, $k=6,7,8$, be the facets of $\mathscr{P}$ adjacent to $P_2$ only at the respective circumscribed grey vertices $v_k$, $k=6,7,8$ in Fig.~\ref{octahedrite9proof}. Let the faces of $P_1$ and $P_2$, that contain a single circumscribed hollow or grey vertex, be $F_k$, $k=3,\dots,8$. Finally, let $P(k)$, $k=6,7,8$, denote the facets adjacent to $P_2$ at $F_k$, $k=6,7,8$, respectively. 

\smallskip
By the convexity argument or by Lemma~\ref{lemma_circuit}, similar to \textbf{D.1}, the facets $\widehat{P}_i$, $i=6,7,8$ can not coincide with the already counted ones, except for $\widehat{P}_j$, $j=3,4,5$ and the facets adjacent only to $P_1$. 

\smallskip
First consider the case when a facet from $\widehat{P}_i$, $i\in\{6,7,8\}$, coincides with $\widehat{P}_j$, $j\in\{3,4,5\}$. Then 

\smallskip
\textbf{1)} either $\widehat{P}_i = \widehat{P}_j$ is such that $(i,j) \neq (7,3)$, $(6,4)$ and $(8,5)$, so the $(2,1)$ circuit $\widehat{P}_jP(i)P_0$ appears;

\smallskip
\textbf{2)} or $\widehat{P}_i = \widehat{P}_j$ has $(i,j) = (7,3), (6,4)$, or $(8,5)$, and contains therefore a part of the geodesic going from $v_i$ to $v_j$ by convexity. Since the edge shared by $F_i$ and $F_j$ belongs to three facets $P_0$, $P_2$ and $P(i)$, then $P(i)$ is adjacent to $P_0$ at $F_j$ and to $P_2$ at $F_i$. Hence $P(i)$ contains the vertices $v_i$, $v_j$ and the geodesic segment between them as well. Since $P(i)$ and $\widehat{P}_i$ have non-intersecting interiors, the two following cases are only possible.

\smallskip
\textbf{2.1)} The geodesic segment $v_iv_j$ belongs to a triangle face of $P(i)$: then $v_iv_j$ is an edge. Observe that the face $F_j$ of $P(i)$ is always a triangle, as in Fig.~\ref{octahedrite9proof}, while the face $F_i$ is either a triangle or a quadrilateral. Then the edges of $F_i$, $F_j$ and the edge $v_iv_j$ constitute a sub-graph in the one-skeleton of $P(i)$. The possible sub-graphs $\tau$ and $\sigma$ depending on the vertex number of $F_i$ are depicted in Fig.~\ref{graphs}. The graph $\tau$ is the one-skeleton of a tetrahedron. The graph $\sigma$ is the one-skeleton of a square pyramid without one vertical edge. By assumption, the facet $P(i)$ is an octahedrite with not more than ten vertices. Such octahedrites are depicted in Fig.~\ref{octahedrite89}-\ref{octahedrite10}, and none of them contains in its one-skeleton a sub-graph combinatorially isomorphic to $\tau$ or~$\sigma$. 

\smallskip
The case when $P(i)$ is an octahedron still remains. Clearly, its one-skeleton does not contain a sub-graph combinatorially isomorphic to $\tau$. However, it contains a sub-graph isomorphic to $\sigma$. The only possible sub-graph embedding of $\sigma$ into the one-skeleton of an octahedron, up to a symmetry, is given in Fig.~\ref{octaembed} on the left. But then the face $F_i$ of $P_2$ correspond to the interior domain $F$ in $P(i)$ coloured grey in Fig.~\ref{octaembed} on the right. Thus, we arrive at a contradiction with the convexity of facets. 

\smallskip
\textbf{2.2)} The geodesic segment $v_iv_j$ belongs to a quadrilateral face of $P(i)$. The general picture of this case is given in Fig.~\ref{quadredge}. Again two sub-graphs $\nu$ and $\omega$ arise, as depicted in Fig.~\ref{graphsedge}. Such sub-graphs appear at most for the octahedrites as given in Fig.~\ref{octahedrite89}-\ref{octahedrite10}. Observe, that none of them contains in its one-skeleton a sub-graph isomorphic to $\nu$. 

\smallskip
All possible embeddings of $\omega$ into the one-skeleton of each considered octahedrite are given, up to a symmetry, in Fig.~\ref{embed1}-\ref{embed8}. Since the edges $e$ and $e^{\prime}$ belong to a single face as in Fig.~\ref{quadredge}, we arrive at a contradiction, since there is no embedding of $\omega$ with this property.  

\smallskip
Finally, consider the case when a facet from $\widehat{P}_i$, $i\in\{6,7,8\}$, coincides with a facet $P^{\prime}$ adjacent only to $P_1$ at a two-face. Then the $(4,0)$ circuit $P_0P_1P^{\prime}P(i)$ arises, in contrary to Lemma~\ref{lemma_circuit}.

\smallskip
So the facets $\widehat{P}_k$, $k=6,7,8$, are different from the already counted ones. Adding them up, we obtain $\mathrm{f}_3(\mathscr{P}) \geq 22 + 3 = 25$.

\medskip 
\textbf{D.3)} Let $P_0 \in \Omega_3(\mathscr{P})$ be the hyperbolic octahedrite with eight vertices depicted on the left in Fig.~\ref{octahedrite89}. Observe that this polyhedron is combinatorially isomorphic to $\mathscr{A}_4$, and hence isometric to it by Andreev's theorem \cite{Andreev2}. Moreover, we suppose that all facets of $\mathscr{P}$ are isometric to $\mathscr{A}_4$, since other possible facet types are already considered in \textbf{D.1} and \textbf{D.2}. 

\smallskip
Consider the facets $P_k$, $k=1,2$, adjacent to the front and the back quadrilateral faces of $P_0$. The facets $P_i$, $i=0,1,2$, are depicted together in Fig.~\ref{octahedrite8proof}, where $P_0$ is coloured grey. We count the facets adjacent to $P_i$, $i=1,2,3$, at faces in Fig.~\ref{octahedrite8proof}. Observe that different numbers on the faces shown in Fig.~\ref{octahedrite8proof} correspond to distinct facets of $\mathscr{P}$ adjacent to them. The counting arguments are completely analogous to those of \textbf{C}. Hence, we obtain the estimate $\mathrm{f}_3(\mathscr{P}) \geq 18$. By taking into account the facets $P_i$, $i=1,2,3$, themselves, it becomes $\mathrm{f}_3(\mathscr{P}) \geq 21$. 

\smallskip
Consider the facets $\widehat{P}_i$, $i=1,2,3,4$, adjacent to $P_2$ only at its circumscribed vertices $v_i$, $i=1,2,3,4$ in Fig.~\ref{octahedrite8proof}. By analogy with the proof in \textbf{D.2}, the $\widehat{P}_i$'s are different from the already counted ones. Thus, we add four new facets and obtain $\mathrm{f}_3(\mathscr{P}) \geq 25$.

\smallskip
Hence, a polytope $\mathscr{P}$ with $\mathrm{f}_3(\mathscr{P}) = 24$ has only octahedral facets and, by the argument from Theorem~\ref{theorem_min_volume}, is isometric to the hyperbolic $24$-cell.~$\square$

\section{A dimension bound for ideal right-angled hyperbolic polytopes}

Given a combinatorial $n$-dimensional polytope $\mathscr{P}$, define the average number of $\ell$-dimensional faces over $k$-dimensional ones as
\begin{equation*}
\mathrm{f}^{\ell}_k(\mathscr{P}) = \frac{1}{\mathrm{f}_k(\mathscr{P})} \sum_{P \in \Omega_k(\mathscr{P})}\mathrm{f}_{\ell}(P).
\end{equation*}
The Nikulin-Khovanski\u{\i} inequality \cite{Khovanskii, Nikulin} applied to the polytope $\mathscr{P}$ which is simple at edges, states that
\begin{equation}\label{NKh}
\mathrm{f}^{\ell}_k(\mathscr{P}) < \binom{n-\ell}{n-k} \frac{\binom{\left\lfloor\frac{n}{2}\right\rfloor}{\ell} + \binom{\left\lfloor\frac{n+1}{2}\right\rfloor}{\ell}}{\binom{\left\lfloor\frac{n}{2}\right\rfloor}{k} + \binom{\left\lfloor\frac{n+1}{2}\right\rfloor}{k}},
\end{equation}
where $\lfloor\circ\rfloor$ means the floor function. 

\medskip
\textsf{\textbf{Corollary (of Theorem~\ref{theorem_min_facet})}} \textit{There are no ideal right-angled hyperbolic polytopes in $\mathbb{H}^n$, if $n \geq 7$.}

\medskip
\begin{proof}
Suppose that $\mathscr{P} \subset \mathbb{H}^n$ is an ideal right-angled hyperbolic polytope, $n\geq 4$. Since we have $\mathrm{f}^3_4(\mathscr{P}) \geq 24$ by Theorem~\ref{theorem_min_facet}, then (\ref{NKh}) implies $n \leq 5$ for $n$ odd and $n \leq 6$ for $n$ even.
\end{proof}

\bigskip
{\it
Alexander Kolpakov

Department of Mathematics

University of Fribourg

chemin du Mus\'ee 23

CH-1700 Fribourg, Switzerland

aleksandr.kolpakov(at)unifr.ch}

\newpage

\begin{figure}
\begin{center}
\includegraphics* [totalheight=4cm]{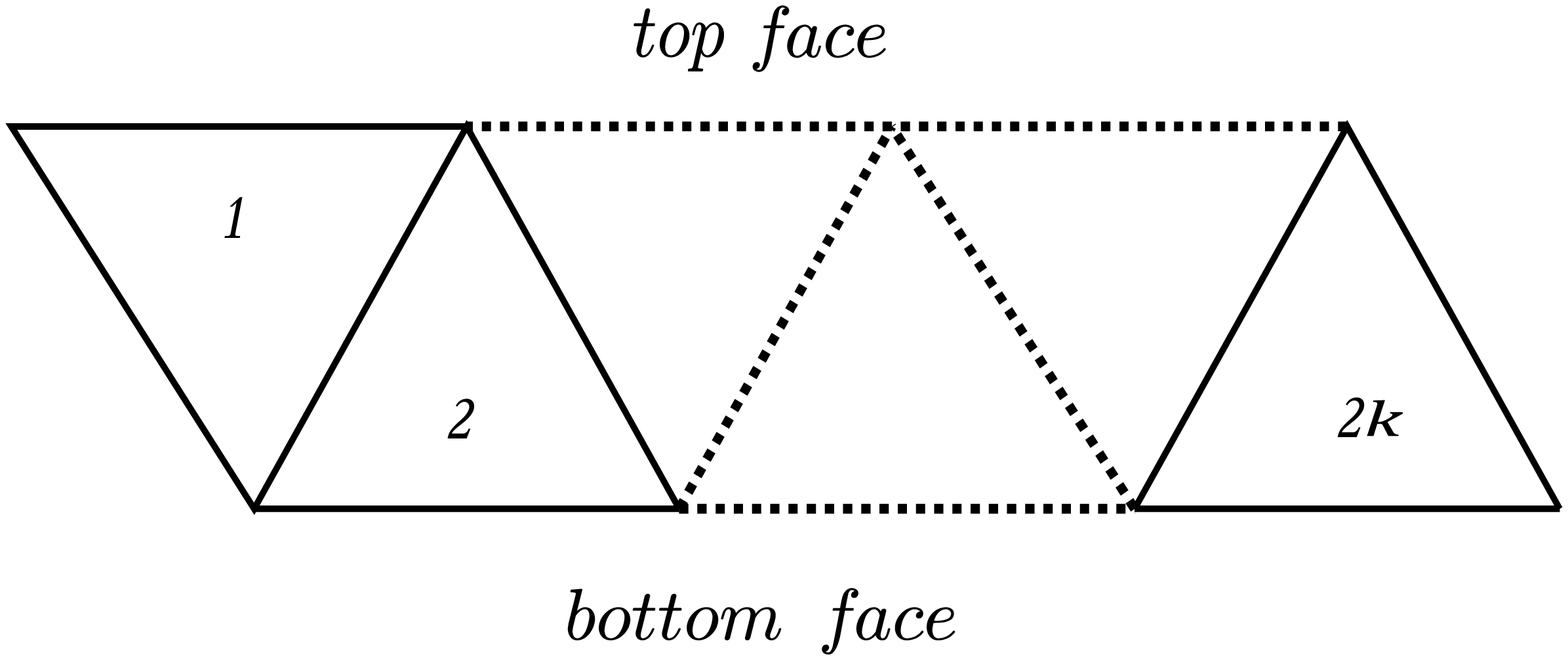}
\end{center}
\caption{Antiprism $\mathscr{A}_k$, $k \geq 3$.} \label{antiprism_k}
\end{figure}

\begin{figure}
\begin{center}
\includegraphics* [totalheight=6cm]{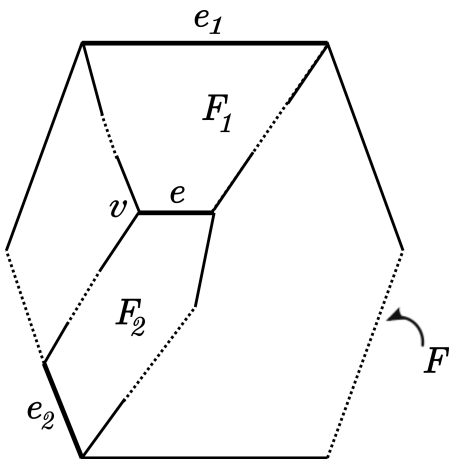}
\end{center}
\caption{Three-circuit deprecated by Andreev's theorem consists of the faces $F$, $F_1$ and $F_2$} \label{antiprism3circuit}
\end{figure}

\begin{figure}
\begin{center}
\includegraphics* [totalheight=6cm]{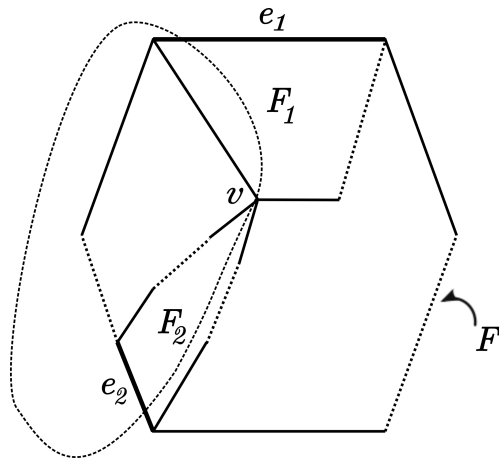}
\end{center}
\caption{The circuit deprecated by Andreev's theorem is indicated by the dashed line} \label{antiprism2circuit}
\end{figure}

\newpage

\begin{figure}
\begin{center}
\includegraphics* [totalheight=6cm]{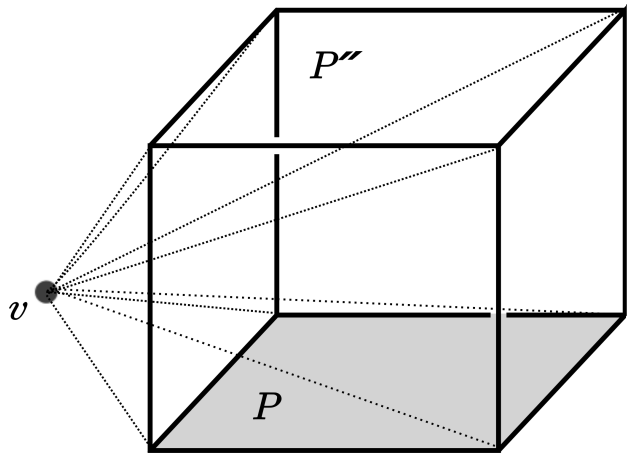}
\end{center}
\caption{The vertex figure $\mathscr{P}_v$} \label{vertex}
\end{figure}

\begin{figure}
\begin{center}
\includegraphics* [totalheight=8cm]{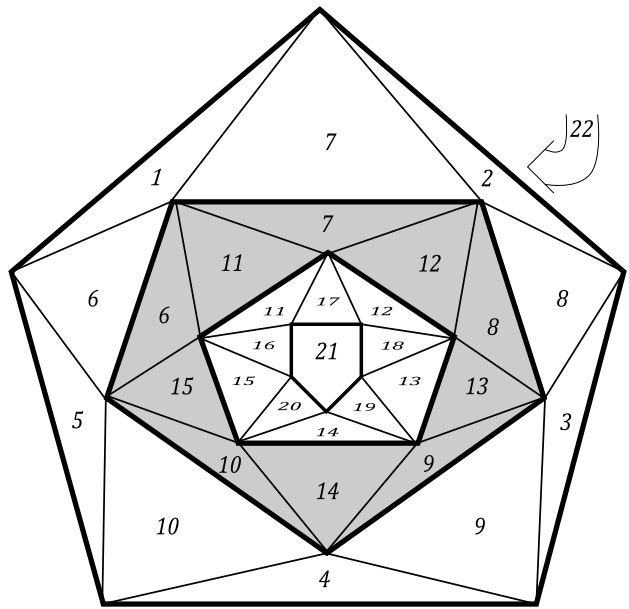}
\end{center}
\caption{Three facets of $\mathscr{P}$ isometric to $\mathscr{A}_5$ and their neighbours} \label{antiprism5proof}
\end{figure}

\begin{figure}
\begin{center}
\includegraphics* [totalheight=5cm]{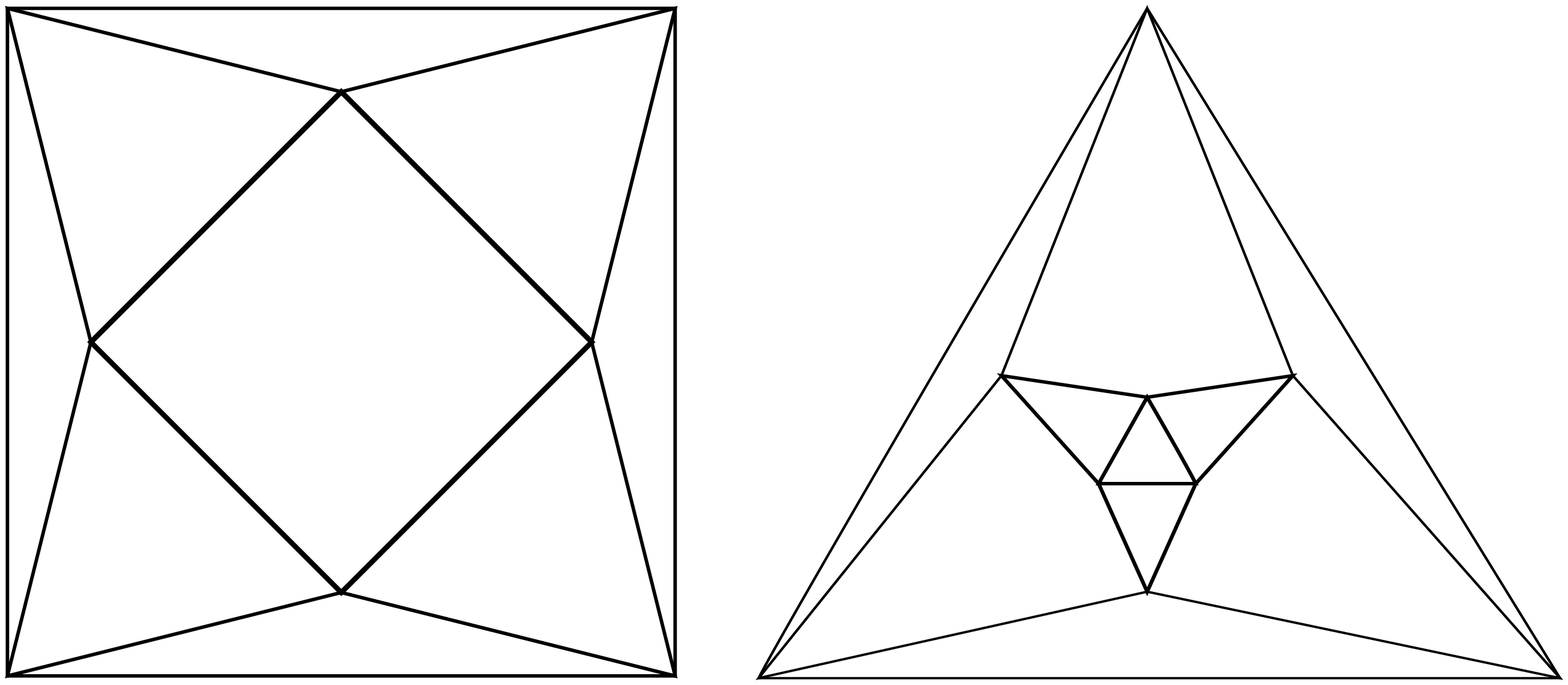}
\end{center}
\caption{Hyperbolic octahedrites with 8 (left) and 9 (right) vertices} \label{octahedrite89}
\end{figure}

\begin{figure}
\begin{center}
\includegraphics* [totalheight=5cm]{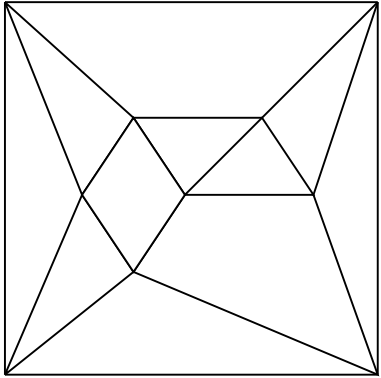}
\end{center}
\caption{Hyperbolic octahedrite with 10 vertices} \label{octahedrite10}
\end{figure}

\newpage

\begin{figure}
\begin{center}
\includegraphics* [totalheight=8cm]{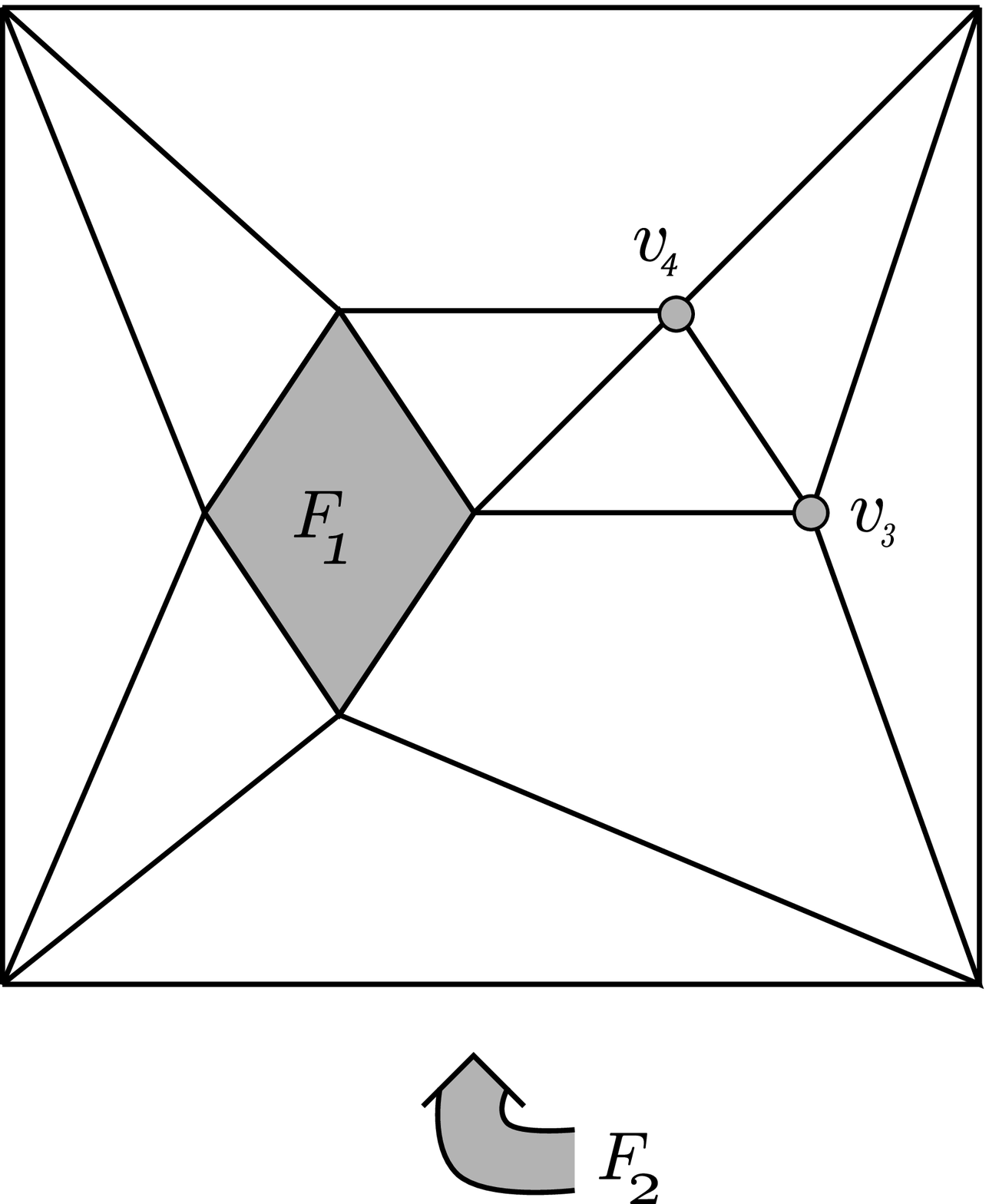}
\end{center}
\caption{Hyperbolic octahedrite with 10 vertices as a facet of $\mathscr{P}$ and its neighbours} \label{octahedrite10proof}
\end{figure}

\begin{figure}
\begin{center}
\includegraphics* [totalheight=8cm]{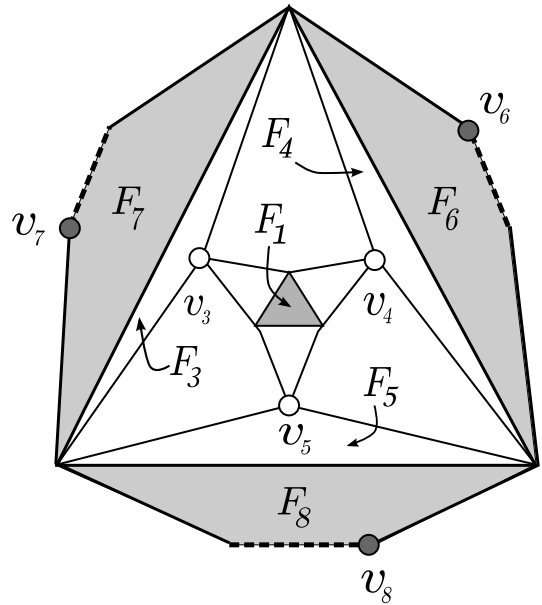}
\end{center}
\caption{Hyperbolic octahedrite with 9 vertices as a facet of $\mathscr{P}$ and its neighbours (omitted edges are dotted)} \label{octahedrite9proof}
\end{figure}

\begin{figure}
\begin{center}
\includegraphics* [totalheight=5cm]{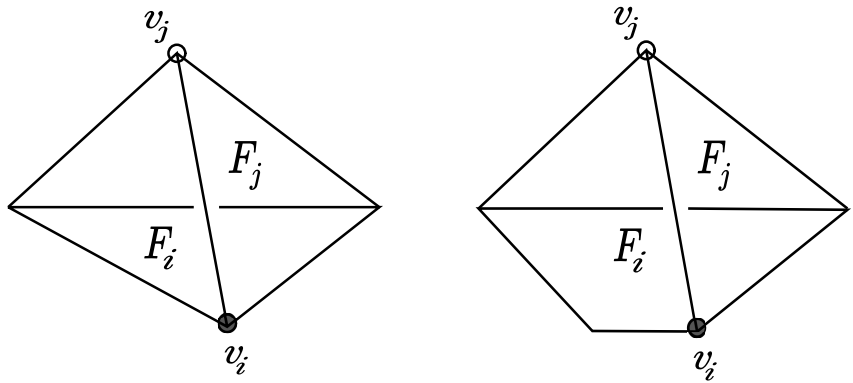}
\end{center}
\caption{Sub-graphs $\tau$ (on the left) and $\sigma$ (on the right)} \label{graphs}
\end{figure}

\begin{figure}
\begin{center}
\includegraphics* [totalheight=5cm]{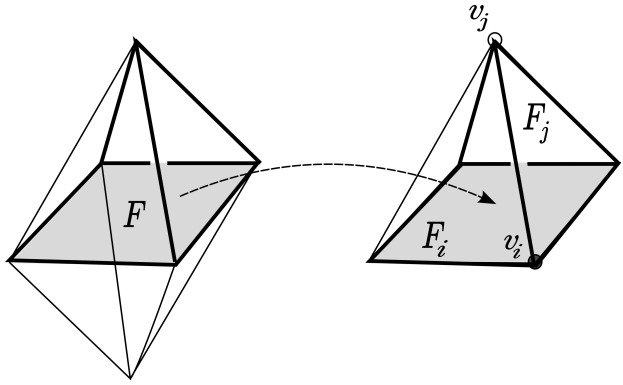}
\end{center}
\caption{Sub-graph $\sigma$ in an octahedron (on the left) and in the facet $P(i)$ (on the right)} \label{octaembed}
\end{figure}

\newpage

\begin{figure}
\begin{center}
\includegraphics* [totalheight=5cm]{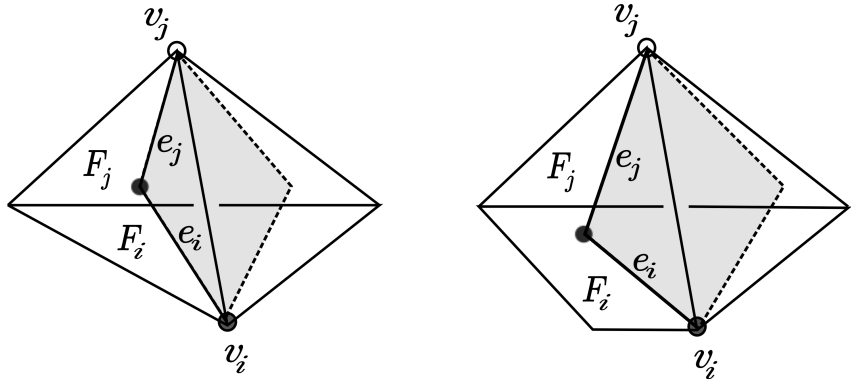}
\end{center}
\caption{The segment $v_iv_j$ belongs to a quadrilateral face} \label{quadredge}
\end{figure}

\begin{figure}
\begin{center}
\includegraphics* [totalheight=5cm]{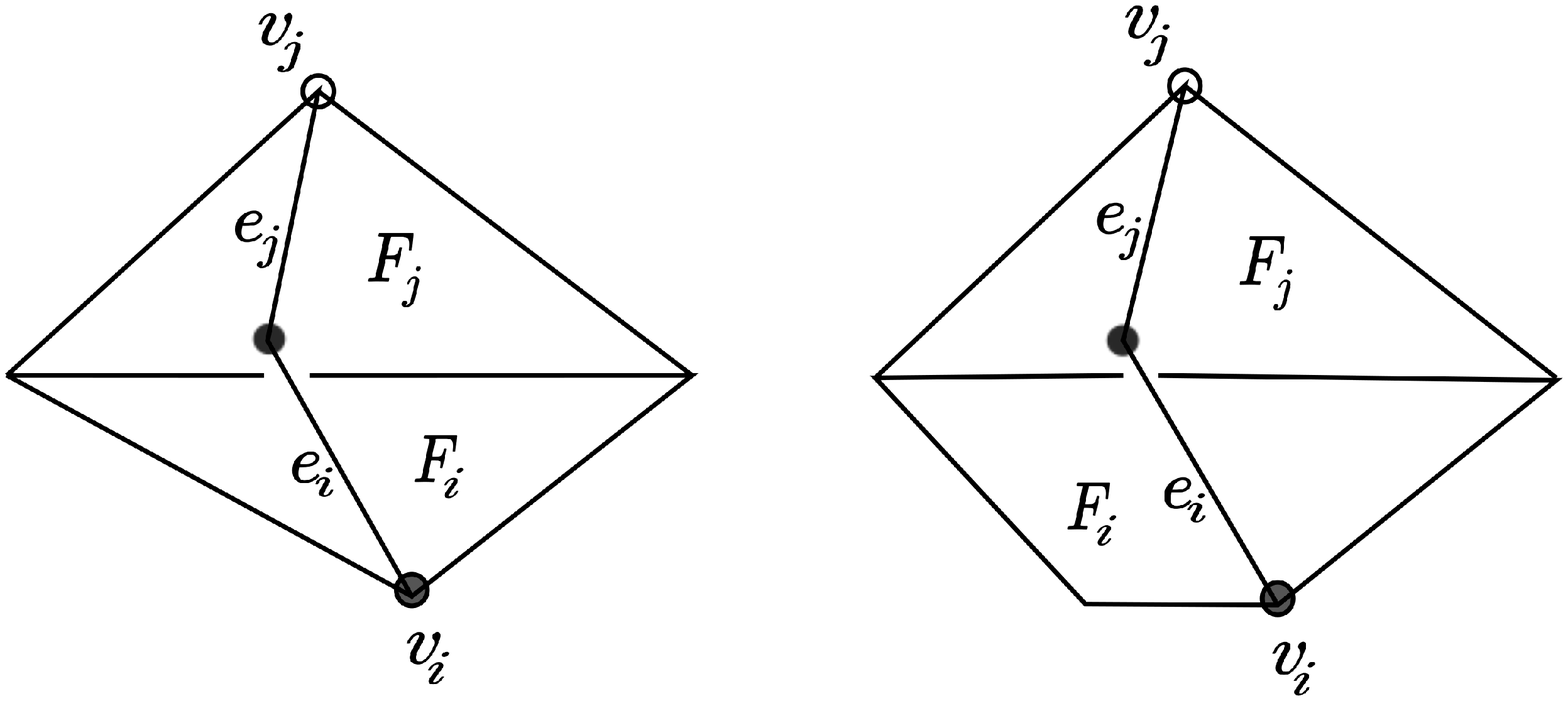}
\end{center}
\caption{Sub-graphs $\nu$ (on the left) and $\omega$ (on the right)} \label{graphsedge}
\end{figure}

\begin{figure}
\begin{center}
\includegraphics* [totalheight=5cm]{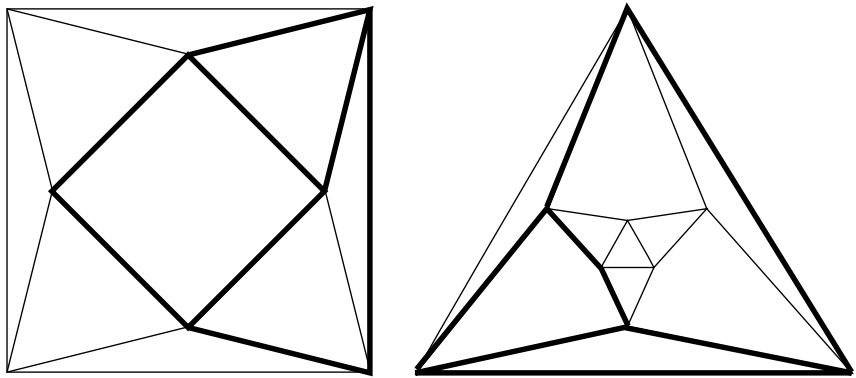}
\end{center}
\caption{Embeddings of the graph $\nu$ into octahedrite facets with $8$ (left) and $9$ (right) vertices} \label{embed1}
\end{figure}

\begin{figure}
\begin{center}
\includegraphics* [totalheight=5cm]{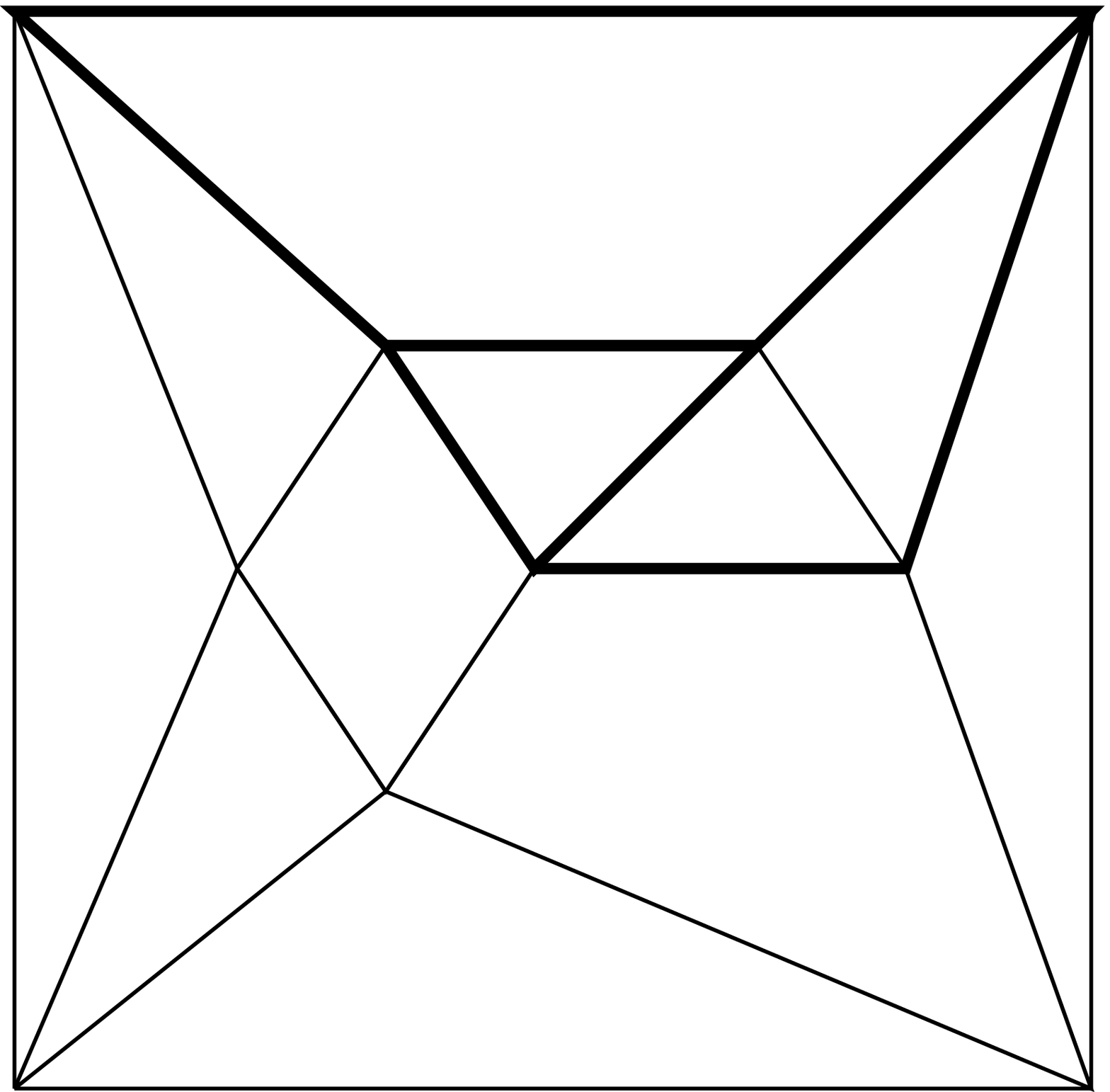}
\end{center}
\caption{An embedding of the graph $\omega$ into the octahedrite facet with $10$ vertices} \label{embed2}
\end{figure}

\newpage

\begin{figure}
\begin{center}
\includegraphics* [totalheight=5cm]{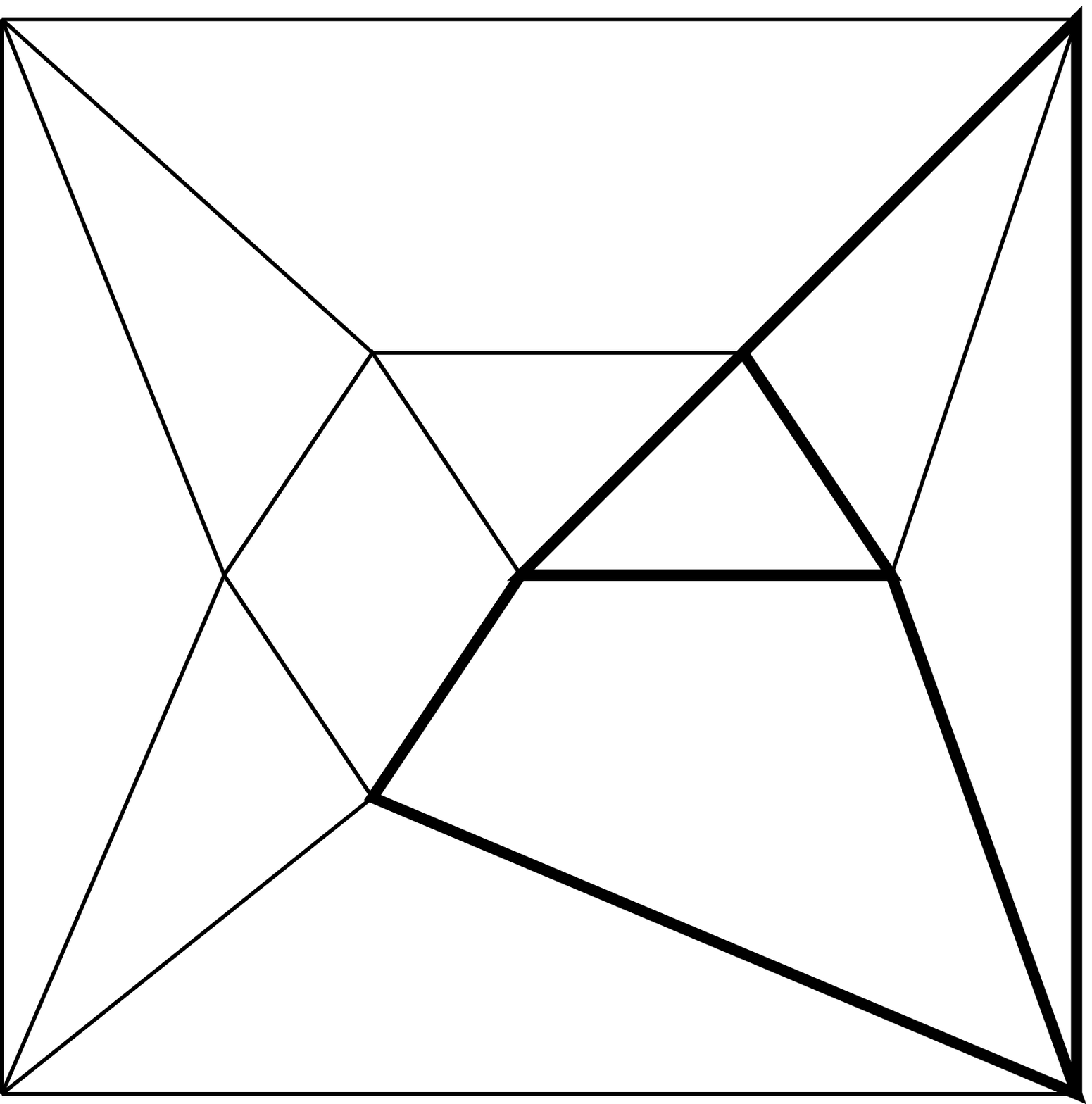}
\end{center}
\caption{An embedding of the graph $\omega$ into the octahedrite facet with $10$ vertices} \label{embed3}
\end{figure}

\begin{figure}
\begin{center}
\includegraphics* [totalheight=5cm]{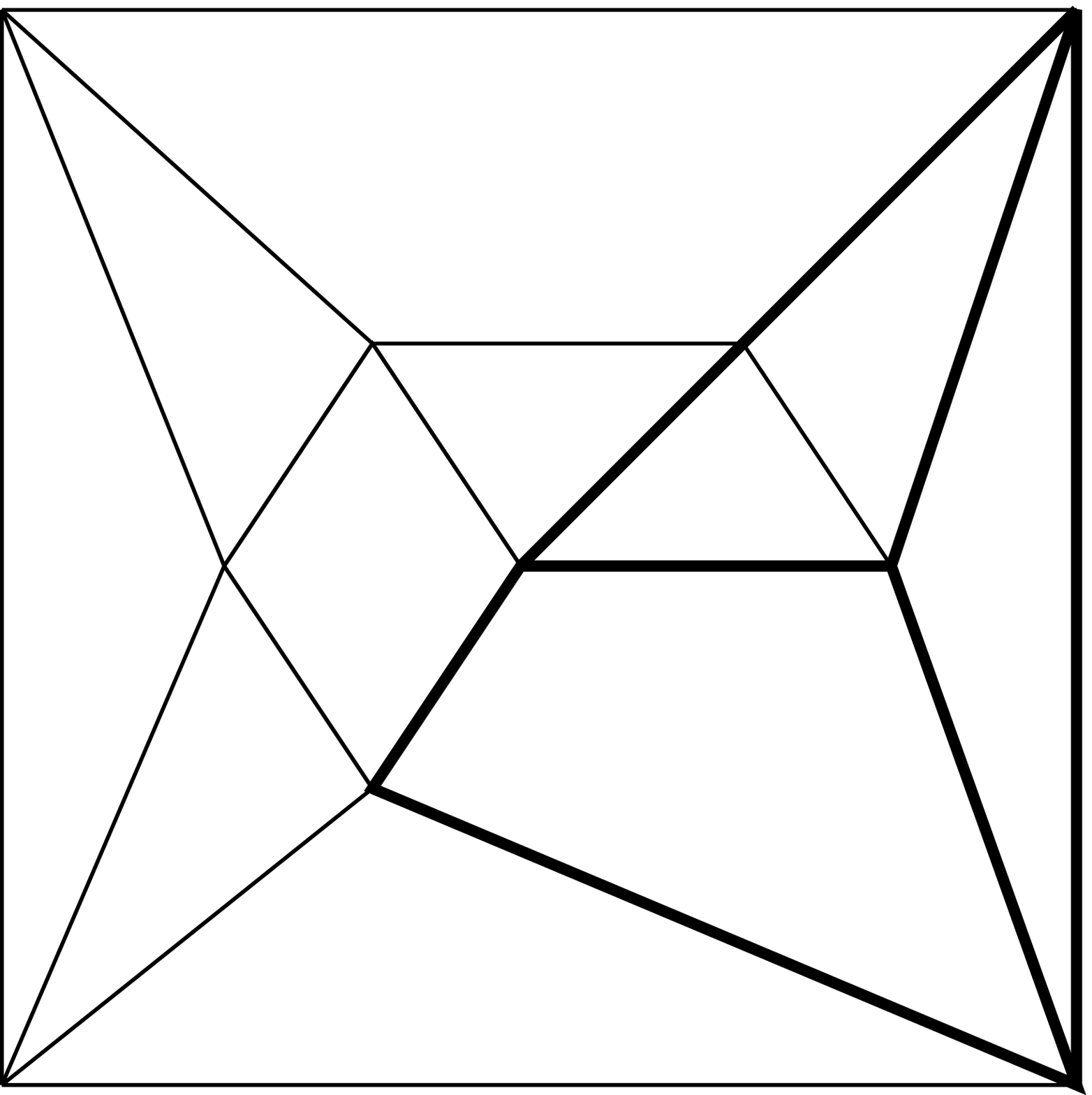}
\end{center}
\caption{An embedding of the graph $\omega$ into the octahedrite facet with $10$ vertices} \label{embed4}
\end{figure}

\begin{figure}
\begin{center}
\includegraphics* [totalheight=5cm]{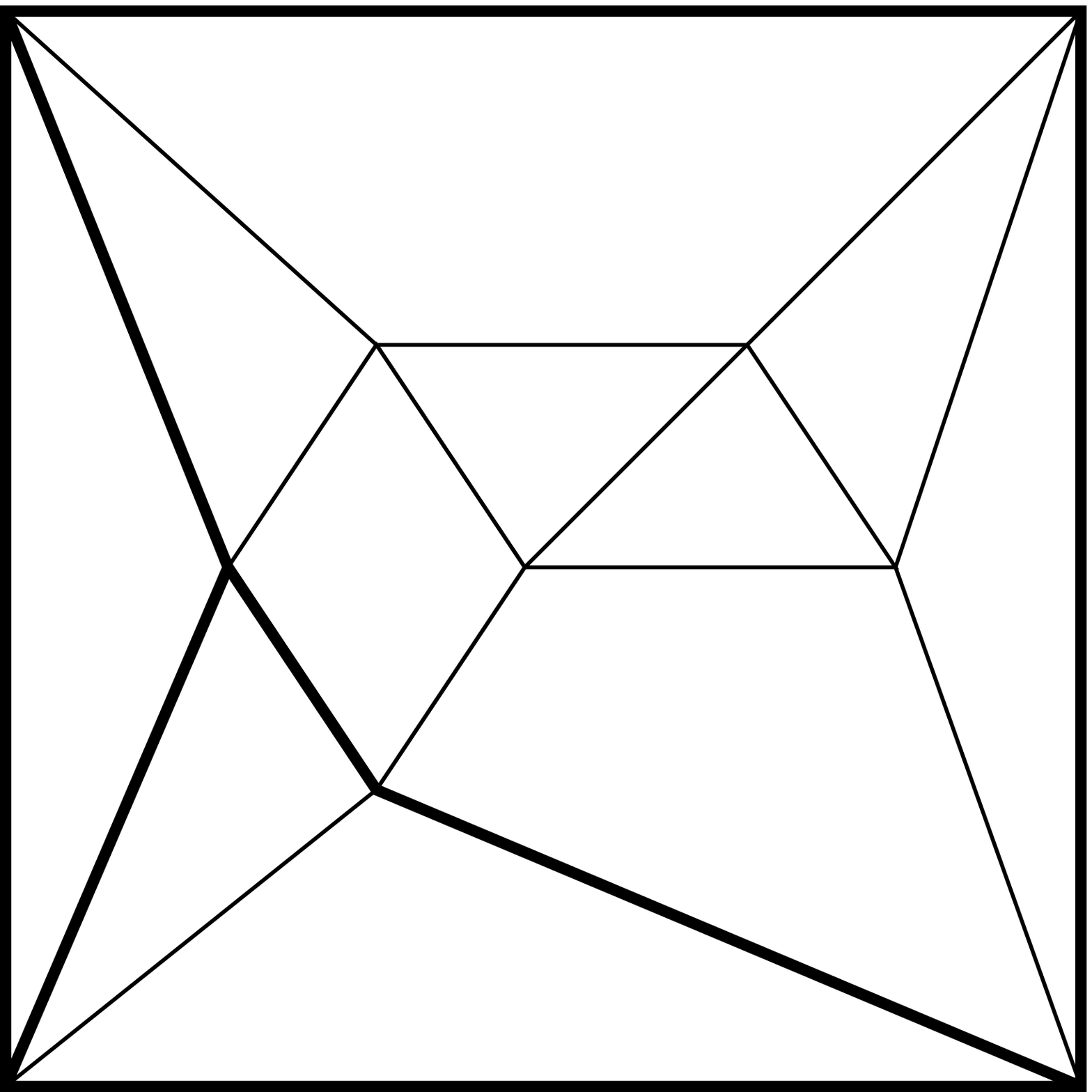}
\end{center}
\caption{An embedding of the graph $\omega$ into the octahedrite facet with $10$ vertices} \label{embed5}
\end{figure}

\begin{figure}
\begin{center}
\includegraphics* [totalheight=5cm]{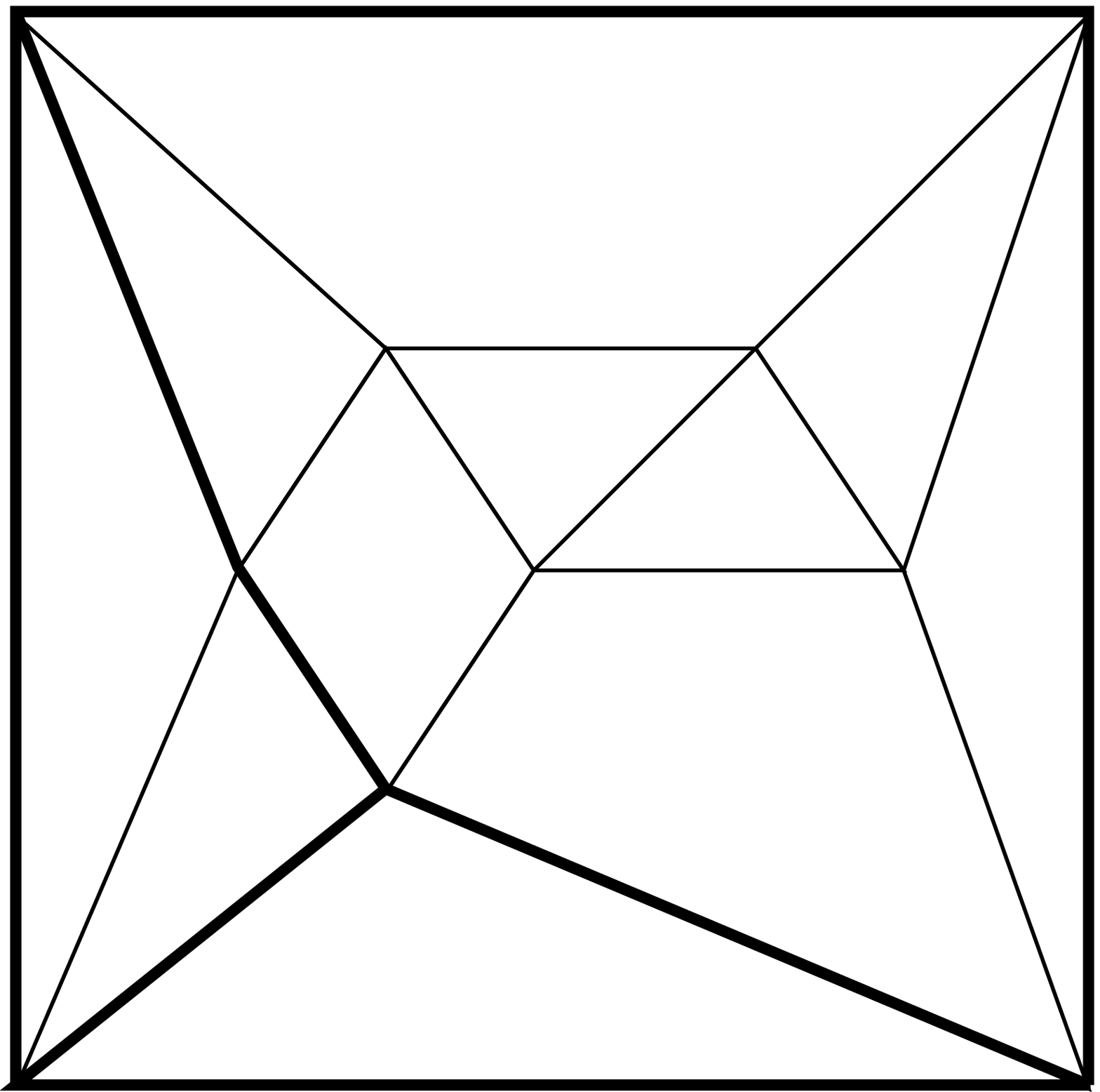}
\end{center}
\caption{An embedding of the graph $\omega$ into the octahedrite facet with $10$ vertices} \label{embed6}
\end{figure}

\begin{figure}
\begin{center}
\includegraphics* [totalheight=5cm]{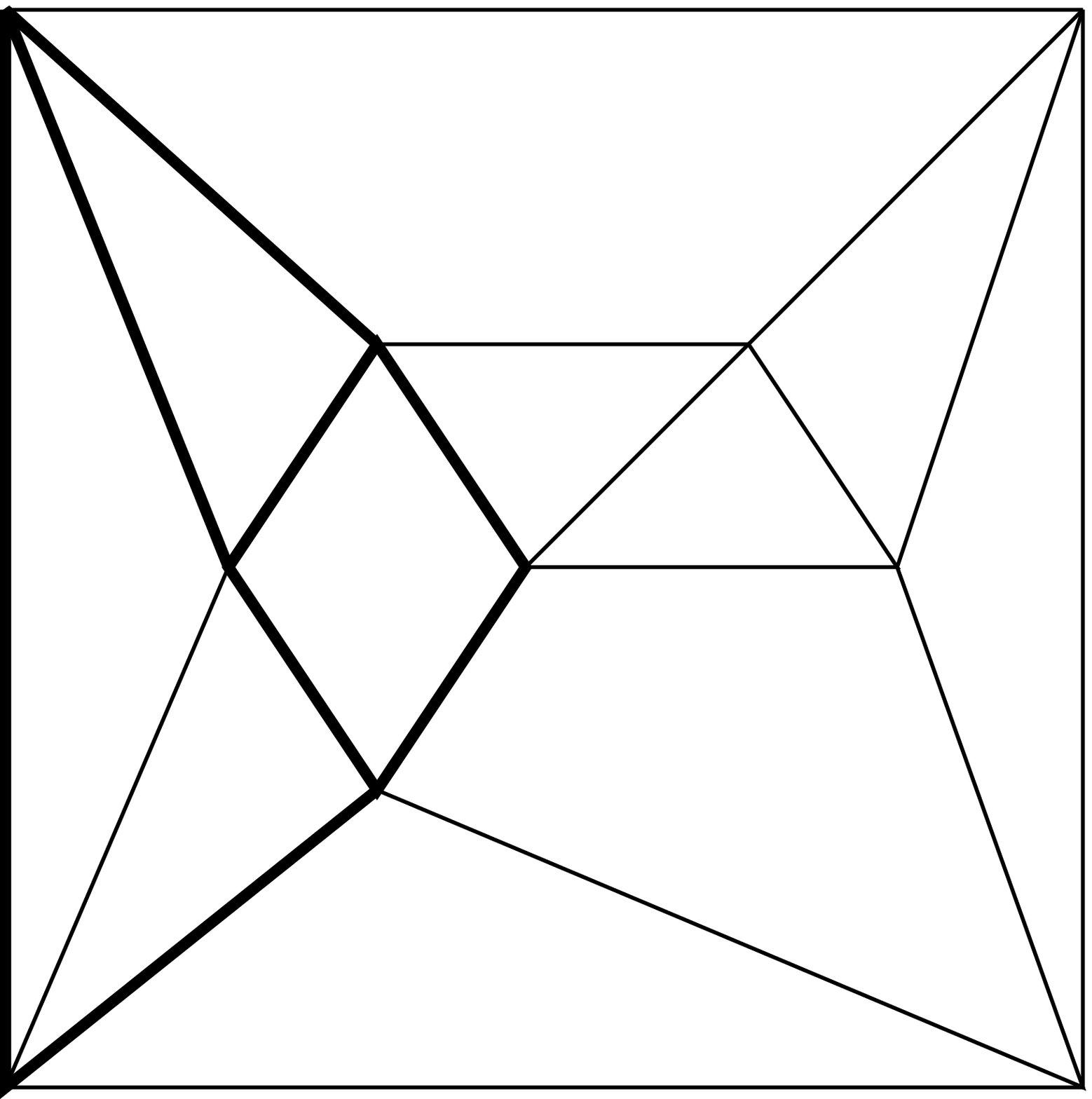}
\end{center}
\caption{An embedding of the graph $\omega$ into the octahedrite facet with $10$ vertices} \label{embed7}
\end{figure}

\newpage

\begin{figure}
\begin{center}
\includegraphics* [totalheight=5cm]{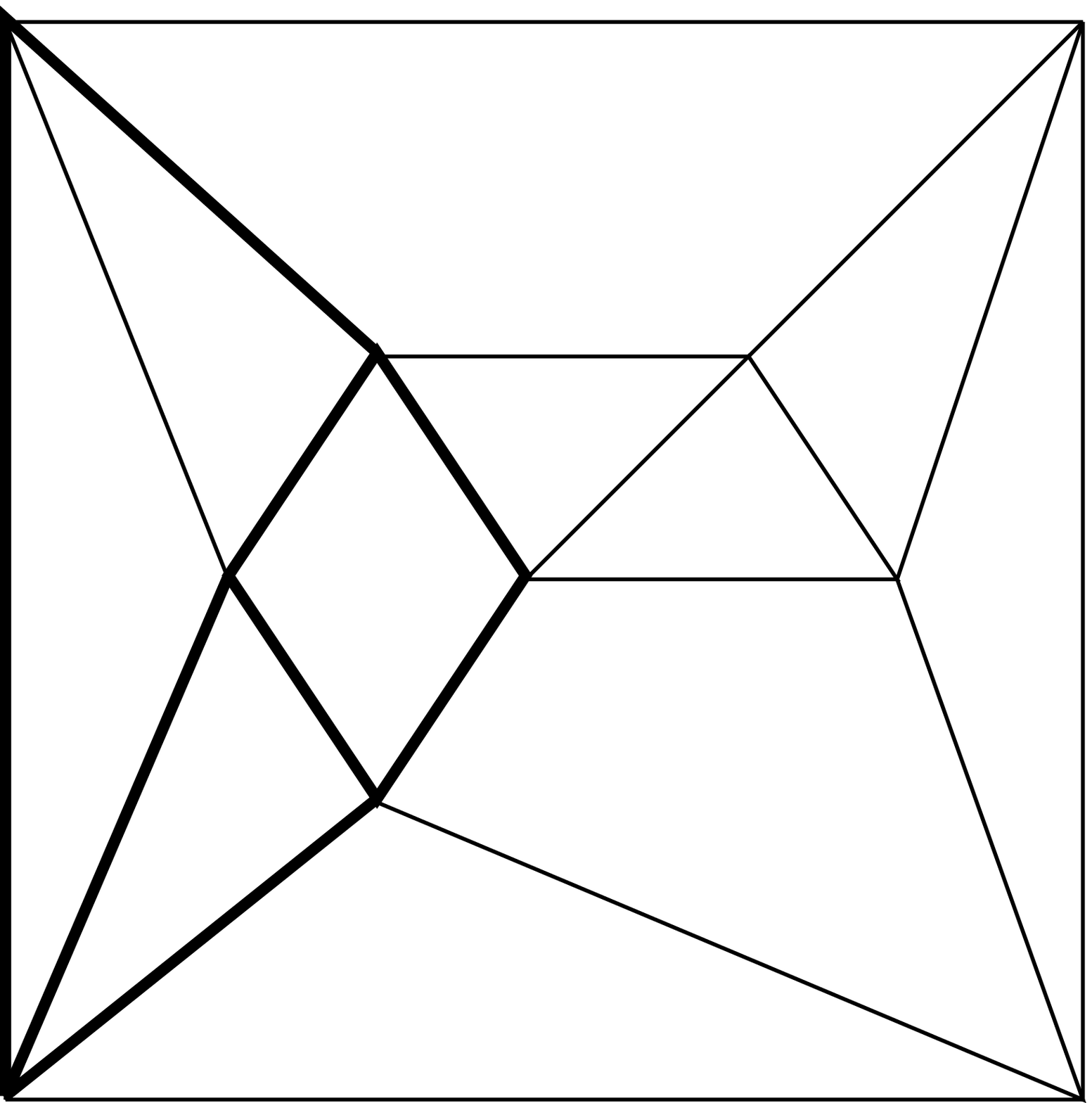}
\end{center}
\caption{An embedding of the graph $\omega$ into the octahedrite facet with $10$ vertices} \label{embed8}
\end{figure}

\begin{figure}
\begin{center}
\includegraphics* [totalheight=8cm]{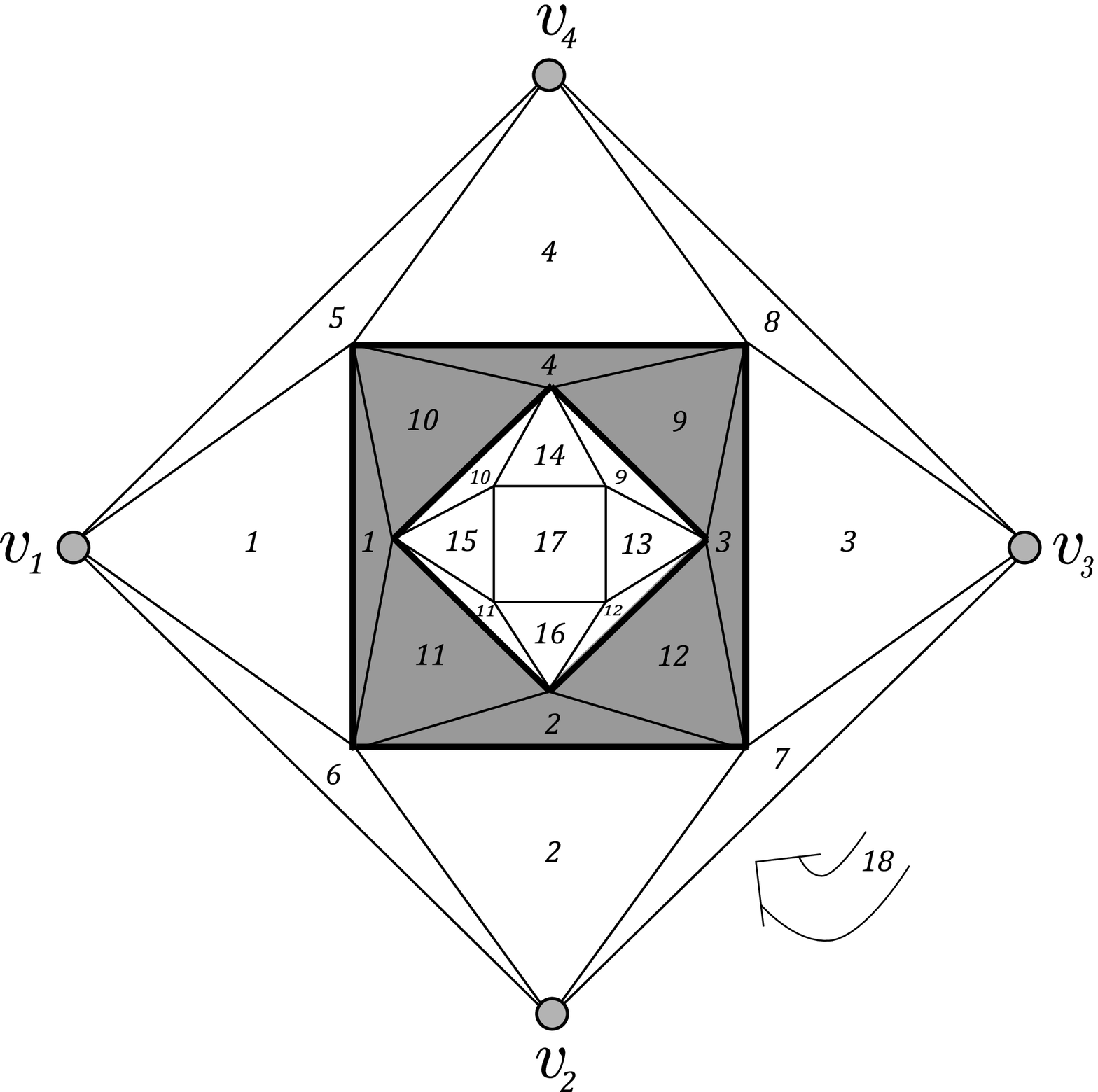}
\end{center}
\caption{Hyperbolic octahedrite with 8 vertices as a facet of $\mathscr{P}$ and its neighbours} \label{octahedrite8proof}
\end{figure}

\end{document}